\documentclass[12pt, a4paper,reqno]{amsart}
\usepackage{amsmath}
\usepackage{amssymb,latexsym}
 \usepackage[latin1]{inputenc}
\usepackage[dvips]{graphicx}
\usepackage{color}
\usepackage{txfonts} 

\newcommand{\R}{\mathbb R}

\newcommand{\N}{\mathbb N}
\newcommand{\C}{\mathbb C}

\newtheorem{theorem}{Theorem} [section]
\newtheorem{lemma}{Lemma} [section]
\newtheorem{proposition}{Proposition} [section]
\newtheorem{corollary}{Corollary} [section]

\let\ssection=\section\renewcommand{\section}{\setcounter{equation}{0}\ssection}
\begin{document}

\title [ Eigenvalues of the sub-Laplace operator]{Inequalities and bounds for the eigenvalues  of the sub-Laplacian on a strictly pseudoconvex CR manifold}
\author{ Amine Aribi }
\address{Amine Aribi : Laboratoire de Math\'{e}matiques et Physique Th\'{e}orique,  UMR-CNRS 6083, Universit\'{e} Fran\c{c}ois Rabelais de Tours, Parc de  Grandmont, 37200 Tours, France} \email{Amine.Aribi@lmpt.univ-tours.fr }

\author{Ahmad El Soufi}
\address{
 Laboratoire de Math\'{e}matiques et Physique Th\'{e}orique,
  UMR-CNRS 6083, Universit\'{e} Fran\c{c}ois Rabelais de Tours, Parc de
  Grandmont, 37200 Tours, France}
  \email{Ahmad.Elsoufi@lmpt.univ-tours.fr }

\thanks{The second author has benefited from the support of the ANR (Agence Nationale de la Recherche) through FOG project ANR-07-BLAN-0251-01.}

\begin{abstract}We establish inequalities for the eigenvalues of the sub-Laplace operator associated with a pseudo-Hermitian structure on a strictly pseudoconvex
 CR manifold. Our inequalities extend those obtained by Niu and Zhang \cite{NiuZhang} for the Dirichlet eigenvalues of the sub-Laplacian on a bounded domain in the Heisenberg group and are in the spirit of the well known   Payne-P\'{o}lya-Weinberger and  Yang universal inequalities.

\end{abstract}

\subjclass[2000]{32V20, 35H20,  58J50.}
\keywords{CR manifold, sub-Laplacian, eigenvalue, Carnot group}

\maketitle
\section{Introduction}
 The sub-Laplacian $\Delta_b$ associated with a pseudo-Hermitian structure on a strictly pseudoconvex CR manifold $M$ is  prototypical of a class of subelliptic operators which appear naturally in several geometric situations that could be gathered under the concept of ``Heisenberg manifolds". The recent work of Ponge \cite{ponge} provides a detailed discussion and a fairly comprehensive presentation of the spectral properties of such operators, including  Weyl asymptotic formulae and heat kernel expansions. 
 
 The sub-Laplacian $\Delta_b$ plays a fundamental role in CR geometry, similar to that played by the Laplace-Beltrami operator in Riemannian geometry (e.g., CR Yamabe problem). Several works published in recent years are devoted to the study of this operator and the investigation of its spectral properties, see for instance \cite{ BarlettaDragomirSpectrum, Barletta, BarlettaDragomirSublaplacians, Changchiu, Greenleaf, LiLuk,  NiuZhang,  ponge}.  In particular, it is known that  $\Delta_b$ is subelliptic of order $\frac 12$, that is for each $x \in M$, there exist a neighborhood $U\subset M$ and a constant $C>0$  such that, $\forall\; u\in C^\infty_0 (U)$,
 $$\|u\|^2_{H^{1/2}}\leq C\langle\left(-\Delta_b +I\right)u, u\rangle_{L^2}.$$ 
This a priori estimate leads to the proof of the  hypoellipticity of $\Delta_b$ and the discreteness of its spectrum when $M$ is a closed manifold 
(see \cite{BarlettaDragomirSpectrum, BarlettaDragomirSublaplacians, MenikoffSjostrand}). Since the pioneering work of Greenleaf  \cite{Greenleaf}, many recent contributions aim to extend to the CR context some of the spectral geometric results established in the Riemannian setting such as Li-Yau or Lichnerowicz-Obata inequalities (see, for example, \cite{BarlettaDragomirSpectrum, Barletta, Changchiu, LiLuk}). It is worth noticing that the determination of the eigenvalues of the sub-Laplacian on the standard CR sphere $\mathbb{S}^{2n+1}$ remains an open problem, except likely for $n=1$ according to \cite {PENGCHEN}.

In this article, we focus on finding bounds on the eigenvalues in the same vein as Payne-P\'{o}lya-Weinberger universal inequalities \cite{PPW}. These inequalities, established in the 1950's for the eigenvalues of the Dirichlet Laplacian in a bounded domain of the Euclidean space $\R^n$, were first stated as follows: for every $k\ge 1$,
\begin{equation}\label{ppw} 
\lambda_{k+1}-\lambda_k\le \frac 4n \left\{\frac 1k\sum_{i=1}^k\lambda_i\right\},
\end{equation}
before being improved by several authors (see for instance \cite{Ashb, hileprotter, Yang}). For example, the following inequality due to Yang \cite{Yang} implies \eqref{ppw} :
\begin{equation}\label{yang} 
\sum_{i=1}^k(\lambda_{k+1}-\lambda_i)^2\le \frac 4n \sum_{i=1}^k\lambda_i(\lambda_{k+1}-\lambda_i).
\end{equation}
Extensions of universal inequalities to bounded domains in Riemannian manifolds other than the Euclidean space have also been obtained.  Let us mention, for example, the  following Yang's type inequality obtained by Ashbaugh \cite{Ashb} for domains of the unit sphere $\mathbb{S}^{n}\subset\R^{n+1}$ (see also \cite{ChengYang05}):
 \begin{equation}\label{ash} 
\sum_{i=1}^k(\lambda_{k+1}-\lambda_i)^2\le \frac 4n \sum_{i=1}^k(\lambda_{k+1}-\lambda_i)(\lambda_i+\frac{n^2}4).
\end{equation}
It is a remarkable fact that the equality holds for every $k$ in this last inequality when the $\lambda_i$ are the eigenvalues of the Laplace-Beltrami operator  on the whole sphere. This  fact was observed by El Soufi, Harrell and Ilias in their paper \cite{SHI} where  inequality \eqref{ash}, as well as many other inequalities in the literature, are recovered as particular cases of the following inequality which applies to the eigenvalues of the Laplace-Beltrami operator of any $n$-dimensional compact Riemannian manifold $M$, with Dirichlet boundary conditions if $\partial M\neq\emptyset$,
 \begin{equation}\label{riem mfd} 
\sum_{i=1}^k(\lambda_{k+1}-\lambda_i)^2\le \frac 4n \sum_{i=1}^k(\lambda_{k+1}-\lambda_i)(\lambda_i+\frac14 {\|H\|_\infty^2}),
\end{equation}
where $H$ is the mean curvature vector field of any isometric immersion of $M$ into a Euclidean space $\R^{n+p}$. Notice that inequality \eqref{riem mfd} had also been found independently by Chen and Cheng \cite{ChenCheng} for the Dirichlet eigenvalues on a bounded domain of a Riemannian  manifold.

    Niu and Zhang \cite{NiuZhang} were certainly the first to address this issue for subelliptic operators. They obtained Payne-P\'{o}lya-Weinberger and Hile-Protter type inequalities for the Dirichlet eigenvalues of the sub-Laplacian on a bounded domain  of the Heisenberg group $\mathbb{H}^n$ of real dimension $2n+1$. The following Yang type inequality has been obtained in this context in \cite{SHI} as an improvement of Niu-Zhang results: 
     \begin{equation}\label{ehi-heis} 
\sum_{i=1}^k(\lambda_{k+1}-\lambda_i)^2\le \frac 2n \sum_{i=1}^k\lambda_i(\lambda_{k+1}-\lambda_i).
\end{equation}
    
    In what follows (see Corollary \ref{corollary4} below), we will prove that inequality \eqref{ehi-heis} remains valid for any strictly pseudoconvex CR manifold $M$ of real dimension $2n+1$ provided it admits a Riemannian submersion over an open set of $\R^{2n}$ which is constant along the characteristic curves of $M$ (i.e. the integral curves of the Reeb vector field). Of course, the standard projection ${\mathbb{H}^{n}}\to \R^{2n}$ satisfies these assumptions. 
    
    As for the CR sphere $\mathbb{S}^{2n+1}$ and domains of $\mathbb{S}^{2n+1}$, we will obtain  the following inequality (Corollary \ref{sphere}): 
      \begin{equation*}
    \sum_{i=1}^k\big(\lambda_{k+1}-\lambda_i\big)^2 \leq \frac {2}n \sum_{i=1}^k \big(\lambda_{k+1}-\lambda_i\big)\big(\lambda_i+n^2\big),
    \end{equation*}
which is sharp for $k=1$. 

All these results are actually particular cases of a more general result (Theorem \ref{them1}) that we establish in Section 3 for an arbitrary strictly pseudoconvex CR manifold $M$ of real dimension $2n+1$ endowed with a compatible pseudo-Hermitian structure $\theta$. Indeed, we prove that the eigenvalues of the sub-Laplacian $\Delta_b$ in a bounded domain $\Omega\subset M$, with Dirichlet boundary conditions if $\Omega\not= M$, satisfy inequalities of the form (see Theorem \ref{them1} for a complete statement): for every integer $k\ge1$ and every $p\in\R$,
\begin{equation}\label{ae-euc1}
    \sum_{i=1}^k\big(\lambda_{k+1}-\lambda_i\big)^p \leq\frac {\max \{2,p\}}n \sum_{i=1}^k \big(\lambda_{k+1}-\lambda_i\big)^{p-1}\big(\lambda_i+ \frac 14 \|H_b(f)\|^2_\infty \big), 
\end{equation}
\begin{equation}\label{ae-euc2}
    \lambda_{k+1}\leq (1+\frac{2}{n})\frac{1}{k}\sum_{i=1}^k\lambda_i+\frac{1}{2n}\|H_b(f)\|^2_\infty ,     
\end{equation}
and
\begin{equation}\label{ae-euc3}
    \lambda_{k+1}\leq (1+\frac{2}{n})k^{\frac{1}{n}}\lambda_1+
    \frac 14\left((1+\frac{2}{n})k^{\frac{1}{n}}-1\right)\|H_b(f)\|^2_\infty,      
\end{equation}
  where $f$ is any $C^2$ semi-isometric map from $(M,\theta)$ to a Euclidean space $\R^m$, and where $H_b(f)$ is a vector field defined similarly to the tension vector field in the Riemannian case (see Section 2 for  definitions). 
   
Besides the CR sphere and Heisenberg groups, many other cases in which one has an explicit expression for $\|H_b(f)\|_\infty$ are given in a series of corollaries in Section 4.

In Section 5 we prove that the inequalities  \eqref{ae-euc1}, \eqref{ae-euc2} and \eqref{ae-euc3} remain true when $f$ is a semi-isometric map from $(M,\theta)$ to a Heisenberg group $\mathbb{H}^m$ which sends the horizontal distribution of $M$ into that of $\mathbb{H}^m$. This can also be seen as a generalization of what was known about the Dirichlet eigenvalues of the sub-Laplacian in a bounded domain of the Heisenberg group, since the identity map of $\mathbb{H}^n$ obviously satisfies $H_b(I_{\mathbb{H}^n})=0$.  

When $M$ is compact without boundary, one has $\lambda_1=0$ and the inequality \eqref{ae-euc3} leads to a relationship between the eigenvalues $\lambda_k$ of the sub-Laplacian of $(M,\theta)$ and the invariant $H_b(f)$ of any semi-isometric map $f$ from $(M,\theta)$ to a Euclidean space. For the first positive eigenvalue $\lambda_2$, we even have the following  inequality 
:
  \begin{equation}\label{ae-reilly}
    \lambda_2(-\Delta_b)\leq \frac{1}{2nV(M,\theta)}\int_M |H_b(f)|_{\R^m}^2,
  \end{equation}
where $V(M,\theta)$ is the volume of $(M,\theta)$. Section 6 deals with these Reilly type inequalities and the characterization of equality cases. For example, we show that the equality holds in \eqref{ae-reilly} if and only if $f(M)$ is contained in a sphere $\mathbb{S}^{m-1}(r)$ of radius $r=\sqrt{\frac {2n}{\lambda_2(-\Delta_b)} }$ and $f$ is a pseudo-harmonic map from $M$ to $S^{m-1}(r)$. 

These Reilly type results are also extended to maps $f$ from $(M,\theta)$ to a Heisenberg group $\mathbb{H}^m$ which sends the horizontal distribution of $M$ into that of $\mathbb{H}^m$ (see Theorem \ref{them4'}).

The last part  of the paper deals with Carnot groups which constitute a natural generalization of Heisenberg groups. A Carnot group is equipped with a natural operator called ``horizontal Laplacian". We give  PPW and Yang type inequalities for the eigenvalues of the horizontal Laplacian in terms of the rank of the horizontal distribution of the group.

\subsubsection*{Acknowledgments} The authors would like to warmly thank  S. Dragomir, N. Gamara, R. Petit and A. Zeghib for useful discussions.

\section{Preliminaries}

Let $M$ be an  orientable CR manifold of CR dimension $n$. This means that $M$ is an  orientable  manifold of real dimension $2n+1$ equipped with a pair $(H(M),J)$, where $H(M)$  is a subbundle  of the tangent bundle $TM$ of real rank $2n$ (often called  Levi distribution) and $J$ is an integrable complex structure on $H(M)$. The integrability condition for $J$ means that, $\forall X,Y\in\Gamma(H(M))$,
$$[X,Y]-[J X,JY] \ \in \ \Gamma(H(M))$$
and
$$[J X,Y]+[X,J Y] = J\left([X,Y]-[J X,JY]\right).$$
Since $M$ is orientable, there exists a nonzero 1-form $\theta\in \Gamma(T^* M)$ such that $Ker\theta = H(M)$. Such a 1-form, called  \emph{pseudo-Hermitian} structure on $M$, is  of course not unique. Actually, the set of pseudo-Hermitian structures that are compatible with the CR-structure of $M$ consists in all the forms $f\theta$ where $f$ is a smooth nowhere zero function on M.

To each  pseudo-Hermitian structure  $\theta$ we associate its \emph{\textbf{Levi form}} $G_{\theta}$ defined on $H(M)$  by
    \begin{equation*}
        G_{\theta}(X,Y)=-d\theta(JX,Y)=\theta([J X,Y])
    \end{equation*}
(note that a factor $\frac 12$ is sometimes put before $d\theta$ so that  in the case of the sphere $\mathbb{S}^{2n+1}\subset \mathbb{C}^{n+1}$, the  Webster metric defined below coincides with the standard metric).
 
The integrability of $J$ implies that $G_{\theta}$ is symmetric and $J$-invariant. The $CR$ manifold $M$ is said to be \textbf{\emph{strictly pseudoconvex}} if the Levi form  $G_{\theta}$ of a compatible pseudo-Hermitian structure $\theta$ is either positive definite or negative definite.  Of course, this condition does not depend on the choice of $\theta$. It implies that the distribution $H(M)$ is far from being integrable.

In all the sequel, a pair $(M,\theta)$ will be called strictly pseudoconvex CR manifold if $M$ is a strictly pseudoconvex CR manifold endowed with a compatible pseudo-Hermitian structure $\theta$ with positive definite Levi form. The structure $\theta$ is then a contact form which induces on $M$ the following volume form  $$\vartheta_{\theta}=\frac{1}{2^n\ n!} \ \theta\wedge (d\theta)^n .$$ 
We will denote by $V(M,\theta)$ the volume of $M$ with respect to $\vartheta_{\theta}$.
    
A pseudo-Hermitian structure $\theta$ on a strictly pseudoconvex CR manifold determines a  vector field $\xi$, often called  characteristic direction or \emph{\textbf{Reeb vector field}} of $\theta$, defined to be the  unique tangent vector field on $M$  satisfying  $\theta(\xi)=1$ and $\xi\rfloor d\theta=0$. Therefore, $L_{\xi}\theta=0$ and $[H(M),\xi]\subset H(M)$.

The \textbf{\emph{Tanaka-Webster connection}} of a strictly pseudoconvex CR manifold $(M,\theta)$ is the unique affine connection $\nabla$ on $TM$ satisfying the following conditions :
\begin{enumerate}
  \item   $\nabla \theta=0$, $\nabla d\theta=0$ and $\nabla J=0$ (hence the distribution $H(M)$ and the vector field $\xi$ are parallel for $\nabla$) 
  \item The Torsion $T_\nabla$ of $\nabla$ is such that, $\forall X,Y\in H(M)$,
$$T_\nabla (X,Y)= -\theta([X,Y])\xi \quad \mbox{and}\quad T_\nabla (\xi,JX)=-J T_{\nabla}(\xi,X) \in H(M).$$
\end{enumerate}

\noindent\textbf{\emph{Basic examples :}}
Standard models for CR manifolds are given by the Heisenberg group and real hypersurfaces of complex manifolds. 
The Heisenberg group will be discussed in Section $5$. If $M$ is an orientable real hypersurface of $\mathbb{C}^{n+1}$, then the sub-bundle  $H(M)$ defined as the orthogonal complement of $J\nu$ in $TM$, where $\nu$ is a unit normal vector field  and $J$ is the standard complex structure of $\mathbb{C}^{n+1}$, is stable by $J$. The pair $(H(M),J)$ endows $M$ with a CR-structure whose compatible pseudo-Hermitian structures are represented by
\begin{equation*}
    \theta(X)=-\frac 1 2 \langle X,J\nu\rangle,
\end{equation*}
where $\langle ,\rangle$ is the standard inner product in $\mathbb{C}^{n+1}$. A straightforward calculation gives
 \begin{equation*}
    G_{\theta}(X,X)=\frac 1 2 \left(B(X,X) +B(JX,JX)\right),
 \end{equation*}
where $B$ is the second fundamental form of the hypersurface. Thus, $M$ is strictly pseudoconvex if and only if the $J$-invariant part of its second fundamental form is positive definite on   $H(M)$.

Since the second fundamental form of the sphere $\mathbb{S}^{2n+1}\subset \mathbb{C}^{n+1}$  coincides with the standard inner product, the above construction endows $\mathbb{S}^{2n+1}$  with a strictly pseudoconvex CR structure whose  Levi form is nothing but the restriction of the standard inner product to the horizontal bundle $H\left(\mathbb{S}^{2n+1}\right)$ where, for every $x\in \mathbb{S}^{2n+1}$, $H_x\left(\mathbb{S}^{2n+1}\right)$ is the orthogonal complement in $\C^{n+1}$ of the complex line passing through $x$.\\

\noindent\textbf{\emph{Sub-Laplacian :}} A Strictly pseudoconvex CR manifold $(M,\theta)$ is equipped with a natural second order differential operator $\Delta_b$ commonly known as the ``sub-Laplacian". This operator is defined in terms of the Tanaka-Webster connection $\nabla$ by:
\begin{equation*}
    \Delta_b u=trace_{G_\theta} \nabla  du. 
\end{equation*}
Given a local $G_{\theta}$-orthonormal frame $\{X_1,...,X_{2n}\}$  of $H(M)$, one has
\begin{equation*}
    \Delta_b u=\sum_{i=1}^{2n}\{X_i\cdot X_i\cdot u-(\nabla_{X_i}X_i).u\} = \sum_{i=1}^{2n}\langle \nabla_{X_i}\nabla^Hu, X_i \rangle_{G_\theta},
\end{equation*}
where $\nabla^Hu\in H(M)$ is the horizontal gradient of $u$ defined by, $\forall X\in H(M)$,  $X\cdot u=G_\theta (X, \nabla^Hu)$.
Integration by parts  yields 
for every compactly supported smooth function $u$ on $M$, 
$$\int_Mu \Delta_b u \ \vartheta_{\theta} =-\int_M |\nabla^Hu|^2_{G_{\theta}} \ \vartheta_{\theta}.$$

When $(M,\theta)$ is strictly pseudoconvex,  the Levi form $G_{\theta}$ extends to a Riemannian metric $g_{\theta}$ on $M$,  sometimes called the Webster metric, so that   the  decomposition $TM=H(M)\oplus \mathbb{R}\xi$ is orthogonal and the vector $\xi$ has unit length, that is, $\forall\; X,Y\in TM$,
   \begin{equation*}
    g_\theta(X,Y)= G_\theta(X^H,Y^H)+\theta(X)\theta(Y),
   \end{equation*}
where  $X^H=\pi_H X$ is the projection of $X$ on $H(M)$ with respect to the  decomposition $TM=H(M)\oplus \mathbb{R}\xi$. Notice that the  Riemannian volume form associated to $g_{\theta}$ coincides with $\vartheta_{\theta}$ (see \cite[Lemma 1]{BarlettaDragomirUrakawa1}). On the other hand,  the Levi-Civita connection  $\nabla^{g_\theta}$ of $(M,g_\theta)$ is related to the Tanaka-Webster connection $\nabla$ by the following identities (see for instance \cite[p.38]{DragomirTomassiniBook}):  for every pair $X$, $Y$ of horizontal vector fields,  $\nabla_X Y = (\nabla^{g_\theta}_X Y)^H$ and, moreover, 
$$  \nabla^{g_\theta}_\xi X- \nabla_\xi X= \frac 12 JX\ , \quad \nabla^{g_\theta}_X \xi-\nabla _X\xi=\nabla^{g_\theta}_X\xi= (\frac 12 J+\tau )X, $$
$$\nabla^{g_\theta}_X Y-\nabla_X Y= -\langle(\frac 12 J+\tau)X,Y\rangle_{g_{\theta}} \xi\ \quad \mbox{and} \quad \nabla^{g_\theta}_\xi\xi=\nabla_\xi\xi=0,$$
where $\tau:H(M)\longrightarrow H(M)$ is the traceless symmetric (1,1)-tensor defined by $\tau X=T_{\nabla}(\xi,X)=\nabla_\xi X -[\xi,X]  $.
Notice that $\tau=0$ if and only if $\xi$ is a Killing vector field w.r.t. the metric $g_\theta$ (and then the metric $g_\theta$ is a Sasakian metric on $M$).

If we denote by $\mbox{div}_{g_\theta} $ the divergence with respect to the metric $g_\theta$, one easily gets
\begin{equation}\label{equa0}
    \Delta_b u=\mbox{div}_{g_\theta} \nabla^H u,
\end{equation}
which immediately  leads 
 to  the following relationship, known as Greenleaf's formula:
$$\Delta_b=\Delta_{g_\theta} - \xi^2$$
where $\Delta_{g_\theta}$ is the Laplace-Beltrami operator of $(M, {g_\theta})$.\\

\noindent\textbf{\emph{Levi tension vector field 
:}} Let $(M,\theta)$ be a strictly pseudoconvex CR manifold of dimension $2n+1$ and let $(N,h)$ be a Riemannian manifold.  The energy  density of a smooth $f:(M,\theta) \longrightarrow (N,h)$ with respect to  horizontal   directions is defined at a point $x\in M$ by
\begin{equation*}
    e_b(f)_x=\frac{1}{2}trace_{G_\theta}(\pi_H f^* h)_x=\frac{1}{2}\sum_{i=1}^{2n}|df(X_i)|^2_h,
\end{equation*}
where $\{X_1,...,X_{2n}\}$ is a local $G_\theta$-orthonormal frame of $H(M)$.
According to \cite[Theorem 3.1]{BarlettaDragomirUrakawa}, the first variation of the  energy functional 
 \begin{equation*}
    E_b(f)=\int_M e_b (f)\vartheta_{\theta}
 \end{equation*}
is determined by the vector, that we will call  ``Levi tension'' of $f$,
\begin{equation*}
    H_b(f)=trace_{G_\theta}\beta_f,
\end{equation*}
where $\beta_f$ is the vector valued 2-form on $H(M)$ given by $$\beta_f(X,Y)=\nabla^f_Xdf(Y)-df(\nabla_X Y),$$  $\nabla^f$ is the connection induced on the bundle $f^{-1}TN$ by the Levi-Civita connection of $(N,h)$, and $\nabla$ is the Tanaka-Webster connection of $(M,\theta)$. That is,
\begin{equation*}
    H_b(f)=\sum_{i=1}^{2n} \nabla_{X_i}^fdf(X_i)-df(\nabla_{X_i}X_i).
\end{equation*}
Mappings with $H_b(f)=0$ are called \emph{pseudo-harmonic} by  Barletta, Dragomir and Urakawa  \cite{BarlettaDragomirUrakawa}. In the case where 
$(N,h)$ is the standard $\mathbb{R}^m$, it is clear that
\begin{equation}\label{equa1}
    H_b (f)=(\Delta_b f_1,...,\Delta_b f_m).
\end{equation}    
    Since $\nabla_X Y = (\nabla^{g_\theta}_X Y)^H=\nabla^{g_\theta}_X Y -\left\langle \left(\frac12 J+\tau\right)X ,Y \right\rangle_{G_\theta}\xi$ for every pair $(X, Y)$ of horizontal vector fields,  one has
$$\beta_f(X,Y)=B_f(X,Y)+\left\langle \left(\frac12 J+\tau\right)X,Y \right\rangle_{G_\theta}df(\xi)$$
and
$$ H_b(f)= H(f)-B_f(\xi,\xi)=H(f)-\nabla^f_\xi df(\xi)$$
where $B_f(X,Y)=\nabla^f_Xdf(Y)-df(\nabla^{g_\theta}_X Y)$ and 
$H(f)=trace_{g_\theta}B_f$ is the tension vector field (see \cite{EellsLemaire}). 
In the particular case where $f$ is an isometric immersion from $(M,g_\theta)$ to $(N,h)$, $B_f$ coincides with the second fundamental form of $f$ and $H(f)$ coincides with its mean curvature vector.

For the natural inclusion $j:\mathbb{S}^{2n+1}\hookrightarrow \mathbb{C}^{n+1}$ of $\mathbb{S}^{2n+1}$,  
the form $\beta_j$ is given by, $\beta_j(X,Y)=-\left\langle X,Y \right\rangle_{\C^{n+1}}\vec{x}+ \left\langle JX,Y \right\rangle_{\C^{n+1}}J\vec{x}$, where $\vec{x}$ is the position vector field (here $\nu(x)=-\vec{x}$ and $\xi(x)=2 J\vec{x}$).  Thus,
\begin{equation}\label{equa2}
    H_b(j)=-2n\ \vec{x}.
\end{equation}

In the sequel we will focus on maps $f:(M,\theta)\longrightarrow (N,h)$ that preserve lengths  in the horizontal directions as well as the orthogonality between $H(M)$ and $\xi$, that is,    $\forall X\in H(M)$,
\begin{equation*}
    |df(X)|_h=|X|_{G_{\theta}} \qquad \mbox{and}\qquad \langle df(X), df(\xi)\rangle_h =0,
\end{equation*}
which  also amounts to  
$f^*h = g_\theta +(\mu-1)\theta^2$
for some  nonnegative function $\mu$ on $M$.
For convenience, such a map  will be termed \emph{\textbf{semi-isometric}}.  
Notice that the dimension of the target manifold $N$ should be at least $2n$. When the dimension of $N$ is $2n$, then a semi-isometric map $f:(M,\theta)\longrightarrow (N,h)$ is noting but a Riemannian submersion satisfying $df(\xi)=0$. Important examples are  given by the standard projection from the Heisenberg group $\mathbb{H}^n$ to $\R^{2n}$ and the Hopf fibration ${\mathbb{S}^{2n+1}}\to{\mathbb{C}P^{n}}$.

\begin{lemma}\label{lem1}
Let $(M,\theta)$ be a strictly pseudoconvex CR manifold and let $(N,h)$ be a Riemannian manifold. If $f:(M,\theta)\longrightarrow (N,h)$ is a $C^2$  semi-isometric map, then the  form $\beta_f$ takes its values in the orthogonal complement of $df(H(M))$. In particular, the vector $H_b(f)$ is orthogonal to $df(H(M))$.
\end{lemma}

\begin{proof} Let $X,Y$ and $Z$ be three horizontal vector fields. Since the Levi-Civita connection of $(N,h)$ is torsionless, one has $\nabla^f_Xdf(Y) - \nabla^f_Ydf(X)= df([X,Y])$. From  the properties of the torsion of the Tanaka-Webster connection $\nabla$,  one has $\nabla_XY-\nabla_YX=[X,Y]^H$. Thus, 
$$\beta_f(X,Y)-\beta_f(Y,X)= \theta ([X,Y])df(\xi).$$
Since $df(\xi)$ is orthogonal to $df(H(M))$, we deduce the following symmetry property:
\begin{equation}\label{equa3'}
 \langle \beta_f(X,Y), df(Z)\rangle_h =\langle \beta_f(Y,X), df(Z)\rangle_h.
\end{equation}
On the other hand, we have, 
\begin{equation}\label{equa3}
    Z\cdot\langle df(X),df(Y)\rangle_h=Z\cdot\langle  X,Y\rangle_{G_{\theta}}.
\end{equation}
Since $G_{\theta}$ is parallel with respect to the Tanaka-Webster connection $\nabla$ and $h$ is parallel with respect to the Levi-Civita connection $\nabla^h$, one gets
\begin{equation*}
    Z\cdot\langle df(X),df(Y)\rangle_h=\langle \nabla^f_Z df(X),df(Y)\rangle_h+\langle df(X),\nabla^f_Zdf(Y)\rangle_h
\end{equation*}
and
\begin{eqnarray*}
  Z\cdot\langle X,Y\rangle_{G_{\theta}} &=&  \langle \nabla_Z X,Y\rangle_{G_{\theta}}+\langle X,\nabla_Z Y\rangle_{G_{\theta}}\\
   &=& \langle df(\nabla_Z X),df(Y)\rangle_h+\langle df(X),df(\nabla_ZY)\rangle_h
\end{eqnarray*}
where the last equality comes from the fact that $\nabla_Z X$ and $\nabla_Z Y$ are horizontal. Replacing into \eqref{equa3} we obtain
\begin{equation*}
    \langle \nabla^f_Z df(X)-df(\nabla_Z X),df(Y)\rangle_h+\langle \nabla^f_Z df(Y)-df(\nabla_ZY),df(X)\rangle_h=0.
\end{equation*}
Therefore, $\forall\; X, Y, Z\in 
H(M),$
\begin{equation}\label{equa4}
    \langle \beta_f(Z,X),df(Y)\rangle_h+\langle \beta_f(Z,Y),df(X)\rangle_h=0.
\end{equation}
Taking $X=Y$ in \eqref{equa4} we obtain, $\forall\; X,Z\in 
H(M),$
\begin{equation}\label{equa4'}
    \langle \beta_f(Z,X),df(X)\rangle_h=0.
\end{equation}
Now, taking $Z=X$ in \eqref{equa4} and using \eqref{equa3'} and \eqref{equa4'}, we get, $\forall\; X,Y\in H(M),$
\begin{equation*}
    \langle \beta_f(X,X),df(Y)\rangle_h=0.
\end{equation*}
The symmetry property \eqref{equa3'} enables us to conclude.
\end{proof}

A direct consequence of Lemma \ref{lem1} is the following
\begin{corollary}\label{cor0}
If $f:(M,\theta)\longrightarrow (N,h)$ is a Riemannian submersion from   a strictly pseudoconvex CR manifold  $(M,\theta)$ to a Riemannian manifold $(N,h)$ with $df(\xi)=0$, then $\beta_f = 0$ and $H_b(f)=0$. 
\end{corollary}

 \section{Eigenvalues of the sublaplacian and semi-isometric maps into Euclidean spaces}

Let $(M,\theta)$ be a strictly pseudoconvex CR manifold and let $\Omega$ be a bounded (relatively compact) domain of $M$. In the case where $M$ is a closed manifold, we allow $\Omega$ to be equal to the whole of $M$. We are interested in Schr\"{o}dinger-type operator $-\Delta_b +V$ where $V$ is a  function on $\Omega$. We assume in all the sequel that  the spectrum of   $-\Delta_b +V$ in $\Omega$,  with Dirichlet boundary conditions if  $\partial \Omega \neq \emptyset$, is discrete and bounded from below.  
We will always denote by $\{\lambda_j\}_{j\ge 1}$ the non decreasing sequence of eigenvalues of $-\Delta_b +V$ and by $\{u_j\}_{j\ge1}$ a complete orthonormal family of eigenfunctions in $\Omega$ with $(-\Delta_b +V) u_j=\lambda_j u_j$.

\begin{theorem}\label{them1}
Let $(M,\theta)$ be a strictly pseudoconvex CR manifold of real dimension $2n+1$  and let $f:(M,{\theta})\longrightarrow \mathbb{R}^m$ be a semi-isometric $C^2$ map. The sequence of eigenvalues $\{\lambda_j\}_{j\ge 1}$ of the Schr\"{o}dinger-type operator $-\Delta_b +V$ in a bounded domain $\Omega \subset M$, with Dirichlet boundary conditions if $\Omega\not= M$, satisfies for every  $k\geq 1$ and $p\in \R$,
  \begin{equation}\label{eqeuc1}
    \sum_{i=1}^k\big(\lambda_{k+1}-\lambda_i\big)^p \leq\frac {\max \{2,p\}}n \sum_{i=1}^k \big(\lambda_{k+1}-\lambda_i\big)^{p-1}\big(\lambda_i+ \frac 14 D_i \big) 
  \end{equation}
with $$D_i=\int_\Omega \left(|H_b(f)|_{\R^m}^2-4 V\right)u_i^2 \ \vartheta_\theta .$$ 
Moreover, if $ V$ is bounded below on $\Omega$, then for every  $k\geq 1$,
 \begin{equation}\label{eqeuc2}
    \lambda_{k+1}\leq (1+\frac{2}{n})\frac{1}{k}\sum_{i=1}^k\lambda_i+\frac{1}{2n}D_{\infty}      
\end{equation}
and
\begin{equation}\label{eqeuc3}
    \lambda_{k+1}\leq (1+\frac{2}{n})k^{\frac{1}{n}}\lambda_1+\frac 14\left((1+\frac{2}{n})k^{\frac{1}{n}}-1\right)D_{\infty}      
\end{equation}
with $D_{\infty}=\sup_\Omega\left(|H_b(f)|_{\R^m}^2-4 V\right)$.
\end{theorem}

Applying this result to the standard CR sphere  whose standard embedding $j:\mathbb{S}^{2n+1}\to \mathbb{C}^{n+1}$ satisfies   $|H_b(j)|^2_{\mathbb{C}^{n+1}}=4n^2$ (see \eqref{equa2}), we get the following 
\begin{corollary}\label{sphere}
Let $\Omega$ be a domain in  the standard CR sphere $\mathbb{S}^{2n+1}\subset \mathbb{C}^{n+1}$. 
The eigenvalues of the operator $-\Delta_b+V$ in $\Omega$, with Dirichlet boundary conditions if $ \Omega\not= \mathbb{S}^{2n+1}$, satisfy, for every  $k\geq 1$ and $p\in\R$,
  \begin{equation*}
    \sum_{i=1}^k\big(\lambda_{k+1}-\lambda_i\big)^p \leq \frac {\max \{2,p\}}n \sum_{i=1}^k \big(\lambda_{k+1}-\lambda_i\big)^{p-1}\big(\lambda_i+n^2-T_i\big)
  \end{equation*}
  with $T_i=\int_\Omega V u_i^2\vartheta_\theta$. Moreover,  if $V$ is bounded below on $\Omega$, then, for every  $k\geq 1$, 
$$\lambda_{k+1}\leq (1+\frac{2}{n})\frac{1}{k}\sum_{i=1}^k \lambda_i+2n-\frac 2n \inf_{\Omega}V$$
and
 $$\lambda_{k+1}\leq (1+\frac{2}{n})k^{\frac{1}{n}}\lambda_1+ C(n,k,V)$$
 with $C(n,k,V)=\left((1+\frac{2}{n})k^{\frac{1}{n}}-1\right)\left(n^2 - \inf_{\Omega} V\right).$

\end{corollary}

Theorem \ref{them1} also applies to the Heisenberg group ${\mathbb{H}^{n}}$ endowed with its standard CR structure. The corresponding sub-Laplacian is nothing but the operator $\Delta_{\mathbb{H}^{n}}=\frac 14\sum_{j\le n}(X_j^2+Y_j^2)$ (see  section 5 for details).  Since the standard projection ${\mathbb{H}^{n}}\to \R^{2n}$ is semi-isometric (up to a dilation, see \eqref{equa9'} below) with zero Levi-tension (see Corollary \ref{cor0}), Theorem \ref{them1} leads to the following corollary which improves the results by Niu-Zhang \cite{NiuZhang} and El Soufi-Harrell-Ilias \cite{SHI}.

\begin{corollary}\label{cor2}
Let $\Omega$ be a domain in  the Heisenberg group ${\mathbb{H}^{n}}$. 
The eigenvalues of the operator $-\Delta_b+V$ in $\Omega$, with Dirichlet boundary conditions, satisfy, for every  $k\geq 1$ and $p\in\R$,
  \begin{equation*}
    \sum_{i=1}^k\big(\lambda_{k+1}-\lambda_i\big)^p \leq \frac {\max \{2,p\}}n \sum_{i=1}^k \big(\lambda_{k+1}-\lambda_i\big)^{p-1}\big(\lambda_i -T_i\big)
  \end{equation*}
  with $T_i=\int_\Omega V u_i^2\vartheta_\theta$. Moreover, if $V$ is bounded below on $\Omega$, then, for every  $k\geq 1$, 
$$\lambda_{k+1}\leq (1+\frac{2}{n})\frac{1}{k}\sum_{i=1}^k \lambda_i-\frac 2n \inf_{\Omega}V$$
and
 $$\lambda_{k+1}\leq (1+\frac{2}{n})k^{\frac{1}{n}}\lambda_1- \left((1+\frac{2}{n})k^{\frac{1}{n}}-1\right)\inf_\Omega V.$$

\end{corollary}

The proof of Theorem \ref{them1} relies on a general result of algebraic nature using commutators. The use of this approach in obtaining bounds for eigenvalues is now fairly prevalent. Pioneering works in this direction are due to Harrell, alone or with collaborators (see \cite{SHI, Harrell1, Harrell2, HarrellMichel, HarrellStubbe1, HarrellStubbe2}). For our purpose, we will use 
the following version that can be found in a recent paper by
Ashbaugh and Hermi \cite{AshbaughHermi} (see inequality $(26)$ of Corollary 3 and inequality $(46)$ of Corollary 8 in \cite{AshbaughHermi}).

\begin{lemma}\label{lma2} Let $A : \mathcal{D} \subset \mathcal{H} \to\mathcal{H}$ be a self-adjoint operator defined
on a dense domain $\mathcal{D}$ which is semibounded below and has a discrete
spectrum $\lambda_1\le \lambda_2\cdots  \le\lambda_i \le \cdots$. Let $B: A(\mathcal{D}) \to\mathcal{H}$
be a symmetric operator which leaves $\mathcal{D}$ invariant. Denoting by 
$\{u_i\}_{i\ge1}$ a complete orthonormal family of eigenvectors of $A$ with $A  u_i=\lambda_i u_i$, we have,
for every   $k\ge 1$ and $p\in\R$, 
  \begin{equation*}
    \sum_{i=1}^k\big(\lambda_{k+1}-\lambda_i\big)^p\big\langle [A,B]u_i, B{u_i}\big\rangle\leq \max\{1,\frac p2\}\sum_{i=1}^k \big(\lambda_{k+1}-\lambda_i\big)^{p-1}\|[A,B]u_i\|^2.
  \end{equation*}

\end{lemma}
\begin{proof}[Proof of Theorem \ref{them1}]
Let $f:(M,\theta)\to \R^m$ be a semi-isometric map and let $f_1,...,f_m$ be its Euclidean  components.  For each  $\alpha=1,...,m$, we denote by $f_\alpha$ the multiplication operator naturally
associated with $f_\alpha$. Let us start by the calculation of $\big\langle [-\Delta_b+V,f_\alpha]u_i,f_\alpha u_i\big\rangle_{L^2}$ and $\|[-\Delta_b+V,f_\alpha]u_i\|^2_{L^2}$. One has, 
 \begin{eqnarray*}
     [-\Delta_b+V,f_\alpha]u_i &=& -\Delta_b(f_\alpha u_i)+f_\alpha (\Delta_b u_i) \\
      &=&  -(\Delta_b f_\alpha)u_i - 2 \langle \nabla^H f_\alpha,\nabla^H u_i\rangle_{G_{\theta}}.
   \end{eqnarray*}
Thus,
\begin{equation}\label{equa5}
    \big\langle [-\Delta_b+V,f_\alpha]u_i,f_\alpha u_i\big\rangle_{L^2}=-\int_\Omega f_\alpha(\Delta_bf_\alpha)u_i^2-\frac{1}{2}\int_\Omega \langle \nabla^H f_\alpha^2,\nabla^H u_i^2\rangle_{G_{\theta}}.
\end{equation}
Here and in  the sequel, all the integrals over $M$ are calculated  with respect to the volume form $\vartheta_\theta$ or, equivalently,  the Riemannian volume element induced by the Webster metric $g_\theta$. The integration over the eventual boundary is calculated with respect to the Riemannian metric induced on $\partial \Omega$ by the Webster metric $g_\theta$.
Integration by parts leads to (see \eqref{equa0})
\begin{eqnarray*}
   \int_\Omega \langle \nabla^H f_\alpha^2,\nabla^H u_i^2\rangle_{G_{\theta}} &=& -\int_\Omega (\Delta_bf_\alpha^2)u_i^2+\int_{\partial M} u_i^2 \langle \nabla^H f_\alpha^2, \nu\rangle_{g_{\theta}} 
\end{eqnarray*}
where $\nu$ is the unit normal vector to the boundary with respect to the Webster metric $g_\theta$. Since $u_i$ vanishes on $\partial \Omega$ when $\partial \Omega\neq \emptyset$, we get
\begin{eqnarray*}
   \int_\Omega \langle \nabla^H f_\alpha^2,\nabla^H u_i^2\rangle_{G_{\theta}} 
   &=&  -\int_\Omega (\Delta_bf_\alpha^2)u_i^2\\
   &=& -2\left[\int_\Omega f_\alpha (\Delta_b f_\alpha) u_i^2+\int_\Omega |\nabla^H f_\alpha|_{G_{\theta}}^2 u_i^2\right].
\end{eqnarray*}
Substituting in \eqref{equa5} we obtain
\begin{equation*}
    \langle [-\Delta_b+V,f_\alpha]u_i,f_\alpha u_i\rangle_{L^2}=\int_\Omega |\nabla^H f_\alpha|_{G_{\theta}}^2 u_i^2.
\end{equation*}
Thus
\begin{equation*}
    \sum_{\alpha=1}^m\langle  [-\Delta_b+V,f_\alpha] u_i,f_\alpha  u_i\rangle_{L^2}= \sum_{\alpha=1}^m\int_{\Omega}|\nabla^H  f_\alpha|_{G_{\theta}}^2u_i^2.
\end{equation*}
Now, since $f$ preserves the Levi-form, one has with respect to a $G_{\theta}$-orthonormal frame $\{e_i\}$ of $H_p(M)$,
\begin{eqnarray*}
   \sum_{\alpha=1}^m|\nabla^H  f_\alpha|^2_{G_{\theta}}&=& \sum_{\alpha=1}^m\sum_{i=1}^{2n}\langle \nabla^Hf_\alpha,e_i\rangle_{G_{\theta}}^2= \sum_{i=1}^{2n}\sum_{\alpha=1}^m\langle \nabla f_\alpha,e_i\rangle_{G_{\theta}}^2 \\
   &=& \sum_{i=1}^{2n}|df(e_i)|^2_{\mathbb{R}^m}= \sum_{i=1}^{2n}|e_i|^2_{G_{\theta}}=2n. 
\end{eqnarray*}
Therefore,
\begin{eqnarray}\label{equa6}
   \sum_{\alpha=1}^m\langle  [-\Delta_b+V,f_\alpha] u_i,f_\alpha  u_i\rangle_{L^2}= 2n\int_{\Omega}u_i^2 =2n.
 \end{eqnarray}
On the other hand, we have
\begin{eqnarray*}
  \|[-\Delta_b+V,f_\alpha]u_i\|^2_{L^2} &=&  \int_{\Omega}\Big((\Delta_b f_\alpha)u_i+2\langle \nabla^H  f_\alpha,\nabla^H  u_i\rangle_{G_{\theta}}\Big)^2 \\
   &=&  \int_{\Omega}(\Delta_b f_\alpha)^2 u_i^2+4\int_{\Omega}\langle \nabla^H  f_\alpha,\nabla ^H  u_i\rangle_{G_{\theta}}^2\\
   &+& 2 \int_{\Omega}(\Delta_b f_\alpha)\langle \nabla^H  f_\alpha,\nabla^H  u_i^2\rangle_{G_{\theta}}.
\end{eqnarray*}
Using \eqref{equa1}, we get
\begin{equation*}
    \sum_{\alpha=1}^m\int_\Omega(\Delta_b f_\alpha)^2 u_i^2=\int_\Omega |H_b(f)|_{\mathbb{R}^m}^2 u_i^2.
\end{equation*}
Using the isometry property of $f$ with respect  to horizontal directions, we get
\begin{eqnarray*}
  \sum_{\alpha=1}^m\langle \nabla^H f_\alpha,\nabla^H u_i\rangle_{G_{\theta}}^2 &=&  \sum_{\alpha=1}^m \langle \nabla f_\alpha,\nabla^H u_i\rangle_{G_{\theta}}^2= \sum_{\alpha=1}^m |df_\alpha (\nabla^H u_i)|_{\mathbb{R}^m}^2 \\
  &=&  |df(\nabla^H u_i)|_{\mathbb{R}^m}^2= |\nabla^H u_i|_{G_{\theta}}^2.
\end{eqnarray*}
Thus,
\begin{equation*}
   \sum_{\alpha=1}^m\int_\Omega\langle \nabla^H  f_\alpha,\nabla ^H  u_i\rangle_{G_{\theta}}^2=\int_\Omega |\nabla ^H  u_i|_{G_{\theta}}^2=\lambda_i-\int_\Omega Vu_i^2.
\end{equation*}
Finally, denoting by $\{E_\alpha\}$ the standard basis of $\mathbb{R}^m$ and using Lemma \ref{lem1}, we get,
\begin{eqnarray*}
  \sum_\alpha^m\int_\Omega\Delta_b f_\alpha\langle \nabla^H f_\alpha, \nabla^H u_i^2\rangle_{G_{\theta}}  &=& \langle \sum_\alpha^m\Delta_b f_\alpha E_\alpha,\sum_\alpha^m\langle \nabla f_\alpha,\nabla^H u_i^2\rangle_{G_{\theta}} E_\alpha \rangle_{\mathbb{R}^m}\\
   &=& \langle H_b(f),df(\nabla^H u_i^2)\rangle_{\mathbb{R}^m}=0.
\end{eqnarray*}
Using all these facts, we get
\begin{equation}\label{equa7}
    \sum_{\alpha=1}^m \|[-\Delta_b+V,f_\alpha]u_i\|^2_{L^2}=4\left(\lambda_i-\int_\Omega Vu_i^2\right)+\int_\Omega  |H_b(f)|_{\mathbb{R}^m}^2 u_i^2.
\end{equation}
Applying  Lemma \ref{lma2}  with $A=-\Delta_b+V$ and $B=f_\alpha$, summing up with respect to $\alpha=1,...,m,$ and  using \eqref{equa6} and \eqref{equa7}, we get the inequality \eqref{eqeuc1}.

\medskip

To prove the inequality \eqref{eqeuc2}, we consider the  quadratic relation that we derive from \eqref{eqeuc1} after replacing $p$ by $2$ and $D_i$ by $D_{\infty}$, that is, $\forall\;k\geq1,$
\begin{equation}\label{equa7'}
   \sum_{i=1}^k\big(\lambda_{k+1}-\lambda_i\big)^2 \leq\frac{2}{n}\sum_{i=1}^k \big(\lambda_{k+1}-\lambda_i\big)\big(\lambda_i+\frac{D_{\infty}}{4} \big)
\end{equation}
which leads to
\begin{eqnarray*}   \lambda_{k+1}^2-\lambda_{k+1}\Big((2+\frac{2}{n})M_k+\frac{1}{2n}D_{\infty}\Big)
+(1+\frac{2}{n})Q_k+\frac{1}{2n}D_{\infty} M_k\leq 0
\end{eqnarray*}
with $M_k=\frac{1}{k}\sum_{i=1}^k\lambda_{i}$ and $Q_k=\frac{1}{k}\sum_{i=1}^k\lambda_i^2$.
Using Cauchy-Schwarz inequality $M_k^2\leq Q_k$, we get
\begin{eqnarray*}   \lambda_{k+1}^2-\lambda_{k+1}\Big((2+\frac{2}{n})M_k+\frac{1}{2n}D_{\infty}\Big)
+(1+\frac{2}{n})M^2_k+\frac{1}{2n}D_{\infty} M_k\leq 0
\end{eqnarray*}
which can also be written as follows:
$$\left(\lambda_{k+1}-M_k\right)\left(\lambda_{k+1}-(1+\frac{2}{n})M_k -\frac{1}{2n}D_{\infty}\right)\le 0.$$
Since $\lambda_{k+1}-M_k$ is clearly nonnegative, we get $\lambda_{k+1}\leq (1+\frac{2}{n})M_k +\frac{1}{2n}D_{\infty}$ which proves \eqref{eqeuc2}.

\medskip

Now, if we set  $\overline{\lambda}_i:=\lambda_i+\frac{1}{4} D_{\infty}$, then the inequality  \eqref{equa7'} reads
 \begin{equation*} \sum_1^k(\overline{\lambda}_{k+1}-\overline{\lambda}_i)^2\leq \frac 2 n \sum_{1}^k(\overline{\lambda}_{k+1}-\overline{\lambda}_i)\overline{\lambda}_i.
   \end{equation*}
Following Cheng and Yang's argument \cite[Theorem 2.1 and Corollary 2.1]{ChengYang},  we obtain
$$\overline{\lambda}_{k+1}\leq \left(1+\frac{2}{n}\right) \overline{\lambda}_1k^{\frac{1}{n}}$$
 which gives immediately the last inequality of the theorem.
\end{proof}

    \section{Applications to Riemannian submersions over submanifolds of the Euclidean space}
Let $(M,\theta)$ be a strictly pseudoconvex CR manifold and let $f:(M,\theta)\to N$ be a Riemannian submersion over a 
Riemannian manifold $N$ of dimension $2n$. The manifold $N$ admits infinitely many isometric immersions into  Euclidean spaces.  For every integer $m\ge 2n$ we denote by  
$\mathcal{I} (N,\R^m)$ the set of all $C^2$ isometric immersions from $N$ to the $m$-dimensional Euclidean space $\mathbb{R}^m$. Thanks to the Nash embedding theorem, the set $\cup_{m\in\N}\mathcal{I} (N,\R^m)$  is never empty,  which motivates the introduction of the following invariant :
$$H^{euc}(N) = \inf_{\phi\in \cup_{m\in\N}\mathcal{I} (N,\R^m)} \|H(\phi)\|_\infty $$
where $H(\phi)$ stands for the mean curvature vector field of $\phi$.  

\begin{theorem}\label{them1'}
Let $(M,\theta)$ be a  strictly pseudoconvex CR manifold of real dimension $2n+1$ and let $f:(M,\theta)\to N$ be a Riemannian submersion over a
Riemannian manifold of dimension $2n$ such that $df(\xi)=0$. The eigenvalues of the operator $-\Delta_b+V$ in a bounded domain $\Omega \subset M$, with Dirichlet boundary conditions if $\Omega\not= M$, satisfy for every  $k\geq 1$ and $p\in\R$,
  \begin{equation}\label{eqsub1}
    \sum_{i=1}^k\big(\lambda_{k+1}-\lambda_i\big)^p \leq\frac {\max\{2,p\}}n \sum_{i=1}^k \big(\lambda_{k+1}-\lambda_i\big)^{p-1}\big(\lambda_i+ \frac 14 H^{euc}(N)^2 -T_i\big) 
  \end{equation}
  with $T_i=\int_\Omega Vu_i^2\vartheta_\theta$. Moreover, if $V$ is bounded below on $\Omega$, then, for every  $k\geq 1$,
  \begin{equation}\label{eqsub2}
  \lambda_{k+1}\leq (1+\frac{2}{n})\frac{1}{k}\sum_{i=1}^k\lambda_i+\frac{1}{2n}H^{euc}(N)^2 -\frac 2n\inf_\Omega V
  \end{equation}
and 
  \begin{equation}\label{eqsub3}
   \lambda_{k+1}\leq (1+\frac{2}{n})k^{\frac{1}{n}}\lambda_1+ C
\end{equation}
with $C=\left((1+\frac{2}{n})k^{\frac{1}{n}}-1\right)\left(\frac 14H^{euc}(N)^2 - \inf_\Omega V\right)$.

\end{theorem}

\begin{proof}
Let $\phi :N\to\R^m$ be any isometric immersion. It is straightforward to check that the map $\hat f=\phi\circ f:(M,\theta)\to \R^m$ is  semi-isometric and that,  $\forall X$, $Y$ $\in H(M)$, 
$$\beta_{\hat f}(X,Y)= d\phi (\beta_f(X,Y)) + B_{\phi}(df(X),df(Y))=B_{\phi}(df(X),df(Y)),$$
where $B_{\phi}$ stands for the second fundamental form of $\phi$ and where the last equality follows from Corollary \ref{cor0}. Now, from the assumptions on $f$, the differential of $f$ induces, for each $x\in M$, an isometry between $H_x(M)$ and $T_{f(x)}N$. Thus,   if $X_1,\cdots, X_{2n}$ is a local orthonormal frame of $H(M)$, then $df(X_1),\cdots, df(X_{2n})$ is also an orthonormal frame of $TN$. This leads to the equality 
$$H_b(\hat f)=H(\phi).$$
Therefore, it suffices to apply Theorem \ref{them1} to $\hat f$ and then take the infimum with respect to $\phi$ to finish the proof.  

\end{proof}

For example, when $N$ is an open set of $\R^{2n}$ or, more generally, a minimal submanifold in $\R^m$, then  $H^{euc}(N)=0$ and the Theorem above gives a class of pseudoconvex CR manifolds including  domains of the Heisenberg group, for which the following holds :

\begin{corollary}\label{corollary4}
Let $(M,\theta)$ be a  strictly pseudoconvex CR manifold of real dimension $2n+1$ which admits a  Riemannian submersion  $f:(M,\theta)\to N$ over a
minimal submanifold  $N$ of dimension $2n$  of  $\R^m$ such that $df(\xi)=0$. 
The eigenvalues of the operator $-\Delta_b+V$ in a bounded domain $\Omega\subset M$, with Dirichlet boundary conditions if $\Omega\not= M$, satisfy for every  $k\geq 1$ and $p\in\R$,
  \begin{equation}\label{eqsub1'}
    \sum_{i=1}^k\big(\lambda_{k+1}-\lambda_i\big)^p \leq\frac {\max\{2,p\}}n \sum_{i=1}^k \big(\lambda_{k+1}-\lambda_i\big)^{p-1}\big(\lambda_i  -T_i\big) 
  \end{equation}
  with $T_i=\int_\Omega Vu_i^2\vartheta_\theta$. Moreover, if $V$ is bounded below on $\Omega$, then for every  $k\geq 1$, 
  \begin{equation}\label{eqsub2'}
  \lambda_{k+1}\leq (1+\frac{2}{n})\frac{1}{k}\sum_{i=1}^k\lambda_i  -\frac 2n\inf_\Omega V
  \end{equation}
and 
  \begin{equation}\label{eqsub3'}
     \lambda_{k+1}\leq (1+\frac{2}{n})k^{\frac{1}{n}}\lambda_1- \left((1+\frac{2}{n})k^{\frac{1}{n}}-1\right)\inf_\Omega V.
\end{equation}

\end{corollary} 

The natural embedding $j:\mathbb{S}^{2n}\to\R^{2n+1}$ of the sphere into the Euclidean space  satisfies $|H(j)|_{\R^{2n+1}}^2=4n^2$. Thus,  Theorem \ref{them1'} leads to the following

\begin{corollary}\label{cor5}
Let $(M,\theta)$ be a strictly pseudoconvex CR manifold of real dimension $2n+1$. Assume that $(M,\theta) $ admits a Riemannian submersion $f:(M,\theta)\to D\subset \mathbb{S}^{2n} $ over a domain $D$ of the standard sphere with $df(\xi)=0$. 
The eigenvalues of the operator $-\Delta_b+V$ in a bounded domain $\Omega\subset M$, with Dirichlet boundary conditions if $\Omega\not= M$, satisfy for every $k\ge1$ and $p\in\R$, 
\begin{equation*}
    \sum_{i=1}^k\big(\lambda_{k+1}-\lambda_i\big)^p \leq \frac {\max \{2,p\}}n \sum_{i=1}^k \big(\lambda_{k+1}-\lambda_i\big)^{p-1}\big(\lambda_i+n^2-T_i\big)
  \end{equation*}
  with $T_i=\int_\Omega V u_i^2\vartheta_\theta$. Moreover, if $V$ is bounded below on $\Omega$, then for every  $k\geq 1$, 
$$\lambda_{k+1}\leq (1+\frac{2}{n})\frac{1}{k}\sum_{i=1}^k \lambda_i+2n-\frac 2n \inf_{\Omega}V$$
and
 $$\lambda_{k+1}\leq (1+\frac{2}{n})k^{\frac{1}{n}}\lambda_1+ C$$
 with $C(n,k,V)=\left((1+\frac{2}{n})k^{\frac{1}{n}}-1\right)\left(n^2 - \inf_{\Omega} V\right).$
 \end{corollary}
 In the particular case of a manifold $M$ without boundary that satisfies the assumptions of Corollary \ref{cor5}, one has, with $V=0$, $\lambda_2(-\Delta_b)=0$,
 $$\lambda_2(-\Delta_b)\le 2n$$
 and, for every $k\ge 1$,
 $$\lambda_{k+1}(-\Delta_b)\leq n(n+2)k^{\frac{1}{n}} -n^2.$$

We denote by  $\mathbb{F}P^{m}$ the $m$-dimensional real projective space if $\mathbb{F}=\R$, the complex projective space of real dimension $2m$ if $\mathbb{F}=\C$, and the quaternionic projective space of real dimension $4m$ if $\mathbb{F}=\mathbb{Q}$. The manifold  $\mathbb{F}P^{m}$  carries a natural metric so that the Hopf fibration $\pi: \mathbb{S}^{d_\mathbb{F}(m+1)-1} \subset \mathbb{F}^{m+1}\to \mathbb{F}P^{m}$  is a Riemannian fibration, where $d_\mathbb{F}= \dim_\R \mathbb{F}$. 

Let $\mathcal{H}_{m+1}(\mathbb{F})=\{ A \in \mathcal{M}_{m+1}(\mathbb{F})\; |\; A^{\ast}:= \overline{^tA}=A \}$  be the vector space of $(m+1)\times(m+1)$ Hermitian matrices with coefficients in $\mathbb{F}$, that we endow with  the inner product
     $$ \langle A,B\rangle= \frac{1}{2} \mbox{trace} (A \, B ).$$
     The map $ \psi: \mathbb{S}^{d_\mathbb{F}(m+1)-1} \subset \mathbb{F}^{m+1}\longrightarrow \mathcal{H}_{m+1}(\mathbb{F})$
given by
$$\psi(z)=
\begin{pmatrix}
|z_{0}|^{2} & z_{0} \bar{z}_{1} & \cdots & z_{0}\bar{z}_{m}\\
z_{1}\bar{z}_{0} & |z_{1}|^{2} & \cdots & z_{1}\bar{z}_{m} \\
 \cdots        & \cdots             & \cdots & \cdots \\
z_{m}\bar{z}_{0}  &  z_{m}\bar{z}_{1} & \cdots & |z_{m}|^{2}
\end{pmatrix}
$$
induces through the Hopf fibration an isometric embedding $\phi$ from $\mathbb{F}P^{m}$ into $\mathcal{H}_{m+1}(\mathbb{F})$. Moreover, $\phi(\mathbb{F}P^{m})$ is a minimal submanifold of the hypersphere $\mathbb{S}\left(\frac{I}{m+1},\sqrt{\frac{m}{2(m+1)}}\right)$ of $\mathcal{H}_{m+1}(\mathbb{F})$ of radius $\sqrt{\frac{m}{2(m+1)}}$ centered at $\frac{I}{m+1}$. One deduces that the mean curvature $H(\phi)$ satisfies 
$$|H(\phi)|^2=2m(m+1)d_\mathbb{F}^2.$$
Therefore, $H^{euc}(\mathbb{F}P^{m})^2 \le 2m(m+1)d_\mathbb{F}^2$ and Theorem \ref{them1'} leads to the following

\begin{corollary}
Let $(M,\theta)$ be a strictly pseudoconvex CR manifold of real dimension $2n+1$ which admits a Riemannian submersion $f:(M,\theta)\to D\subset \mathbb{F}P^{m}$ over a domain of the projective space $\mathbb{F}P^{m}$ of real dimension 2n (i.e. $m=2n/d_\mathbb{F} $) with $df(\xi)=0$. 
The eigenvalues of the operator $-\Delta_b+V$ in a bounded domain $\Omega\subset M$, with Dirichlet boundary conditions if $\Omega\not= M$, satisfy for every  $k\geq 1$ and $p\in\R$,
\begin{equation*}
    \sum_{i=1}^k\big(\lambda_{k+1}-\lambda_i\big)^p \leq \frac {\max \{2,p\}}n \sum_{i=1}^k \big(\lambda_{k+1}-\lambda_i\big)^{p-1}\big(\lambda_i+n(2n+d_\mathbb{F})-T_i\big)
  \end{equation*}
  with $T_i=\int_\Omega V u_i^2\vartheta_\theta$. Moreover, if $V$ is bounded below on $\Omega$, then for every  $k\geq 1$, 
$$\lambda_{k+1}\leq (1+\frac{2}{n})\frac{1}{k}\sum_{i=1}^k \lambda_i+2(2n+d_\mathbb{F})-\frac 2n \inf_{\Omega}V$$
and
 $$\lambda_{k+1}\leq (1+\frac{2}{n})k^{\frac{1}{n}}\lambda_1+ C$$
 with $C(n,k,V)=\left((1+\frac{2}{n})k^{\frac{1}{n}}-1\right)\left(n(2n+d_\mathbb{F}) - \inf_{\Omega} V\right).$
 \end{corollary}

    \section{Eigenvalues of the sub-laplacian and semi-isometric maps into  Heisenberg groups}

A model for the Heisenberg group is given by $\mathbb{H}^m=\R^{2m+1}\cong\mathbb{C}^m\times \mathbb{R}$  endowed with the group  law
\begin{equation*}
    (z, t) \cdot (w, s) = (z + w, t + s + 2Im \langle z,w\rangle ),
\end{equation*}
where $(z,t)=(z^1,...,z^n,t)$,  $(w,s)=(w^1,...,w^n,s)\in \mathbb{C}^m\times \mathbb{R}$, and $\langle z,w\rangle=\sum_{j\le m}z^j\overline{w}^j$ is the standard complex scalar product in $\mathbb{C}^m$. A natural basis of the corresponding Lie algebra is given by the family of left-invariant vector fields $\{X_1,...,X_m,Y_1,...,Y_m, T\}$  that coincides with the standard basis of $\R^{2m+1}$ at the origin.
 That is, $T=\frac{\partial}{\partial t}$ and, $\forall j \leq m$,
$$X_{j}=\frac{\partial}{\partial x_{j}}+2y_{j}\frac{\partial}{\partial t},\;\;\; Y_{j}=\frac{\partial}{\partial y_{j}}-2x_{j}\frac{\partial}{\partial t}.$$
The Levi distribution $H(\mathbb{H}^m)$ is spanned by the vector fields $\{X_j, Y_j\}_{j\le m}$. The complex sub-bundle   $T^{1,0}$ of $T\mathbb{H}^m\otimes\C$ spanned by 
$$Z_{j}=\frac{\partial}{\partial z_{j}}+ i \bar{z}_j\frac{\partial}{\partial t}=\frac 12 \left(X_j-i Y_j\right), \qquad j=1,\dots m$$
is such that $H(\mathbb{H}^m)=\mbox{Re}\left(T^{1,0}\oplus T^{0,1}\right)$, with $$T^{0,1}=\mbox{span}\left\{\bar{Z}_{j}=\frac{\partial}{\partial \bar{z}_{j}}- i {z}_j\frac{\partial}{\partial t}=\frac 12 \left(X_j+i Y_j\right), \  j=1,\dots m\right\}.$$ 
This endows $H(\mathbb{H}^m)$ with an almost complex structure $J$ (so that $T^{1,0}=\ker (J-i)$ and $T^{0,1}=\ker (J+i)$) which is actually integrable since $[Z_j,Z_k]=0$ for all $j,k \le m$. Moreover, we have for all $j\le m$,  $JX_j=Y_j$.

The standard pseudo-Hermitian structure on $\mathbb{H}^m$ is
\begin{equation}\label{equa8}
    \theta_{\mathbb{H}^{m}}=dt+i\sum_{j=1}^m(z^jd\bar{z}^j-\bar{z}^jdz^j)=dt+2\sum_{j=1}^m(x^jdy^j-y^jdx^j),
\end{equation}
whose differential is $d\theta_{\mathbb{H}^{m}}=2i\sum_{j=1}^n dz^j\wedge d\bar{z}^j$ and characteristic direction is $T=\frac{\partial}{\partial t}$. Since, for all  $j\le m$ and $k\le m $, one has $[X_j,Y_k]=-4 \delta_{jk} T$ and $[X_j,X_k]=[Y_j,Y_k]=0$, the Levi form $G_{ \theta_{\mathbb{H}^{m}}}$ on $H(\mathbb{H}^m)$ satisfies
$$ G_{ \theta_{\mathbb{H}^{m}}}(X_j,X_k)=G_{ \theta_{\mathbb{H}^{m}}}(Y_j,Y_k)=4\delta_{jk} \qquad \mbox{and}\qquad G_{ \theta_{\mathbb{H}^{m}}}(X_j,Y_k)=0.$$
We will denote by $g_{\mathbb{H}^m}$ the corresponding Webster metric. 

For a vector $W\in T_{(z,t)}\mathbb{H}^m$, 
if we denote by $\{v_1,w_1,...,v_n,w_n,s\}$ its components with respect to the standard basis of $\R^{2m+1}$, i.e.,
\begin{equation*}
    W=\sum_{j=1}^m v_j\frac{\partial}{\partial x_j}+w_j\frac{\partial}{\partial y_j}+s\frac{\partial}{\partial t},
\end{equation*}
then
\begin{eqnarray}\label{equa9}
\nonumber W  &=& \sum_j( v_j X_j + w_j Y_j) + \{s+2\sum_j w_j x_j-v_j y_j\}\frac{\partial}{\partial t}\\
   &=& \sum_j( v_j X_j + w_j Y_j) +\theta_{\mathbb{H}^{m}}(W)T.
\end{eqnarray}
Hence, the coordinates of $W$ with respect to the basis $(X_1,Y_1,...,X_m,Y_m,T)$ of $T_{(z,t)}\mathbb{H}^m$, are $\{v_1,w_1,...,v_n,w_n,\theta_{\mathbb{H}^{m}}(W)\}$. Thus,
\begin{eqnarray}\label{equa9'}
  g_{\mathbb{H}^m}(W,W) &=& 4\sum_j^n(v_j^2+w_j^2)+\theta_{\mathbb{H}^{m}}(W)^2 \\
   &=& 4 |W|^2_{\R^{2m+1}}-4 s^2+\theta_{\mathbb{H}^{m}}(W)^2.
\end{eqnarray}
In particular, if $W$ is horizontal, then $g_{\mathbb{H}^{m}}(W,W)=4|W|^2_{\R^{2m+1}}-4s^2.$

\begin{theorem}\label{them2}
    Let $(M,\theta)$ be a strictly pseudoconvex CR manifold of dimension $2n+1$ and let $f:M\longrightarrow \mathbb{H}^{m}$  be a  $C^2$ semi-isometric map satisfying  $df (H(M))\subseteq H(\mathbb{H}^{m})$. 
  Then the eigenvalues of the operator $-\Delta_b+V$ in any bounded domain $\Omega\subset M$, with Dirichlet boundary conditions if $\Omega\not= M$, satisfy  for every  $k\geq 1$ and $p\in\R$,
   \begin{equation}\label{eqheis1}
    \sum_{i=1}^k\big(\lambda_{k+1}-\lambda_i\big)^p \leq\frac {\max \{2,p\}}n \sum_{i=1}^k \big(\lambda_{k+1}-\lambda_i\big)^{p-1}\big(\lambda_i+ \frac 14 D_i \big) 
  \end{equation}
with $$D_i=\int_\Omega \left(|H_b(f)|_{\mathbb{H}^{m}}^2-4 V\right)u_i^2 \ \vartheta_\theta .$$ 
Moreover, if $V$ is bounded below on $M$, then for every  $k\geq 1$,
   \begin{equation}\label{eqheis2}
    \lambda_{k+1}\leq (1+\frac{2}{n})\frac{1}{k}\sum_{i=1}^k\lambda_i+\frac{1}{2n}D_{\infty}      
\end{equation}
and
\begin{equation}\label{eqheis3}
    \lambda_{k+1}\leq (1+\frac{2}{n})k^{\frac{1}{n}}\lambda_1+\frac 14\left((1+\frac{2}{n})k^{\frac{1}{n}}-1\right)D_{\infty}      
\end{equation}
with $D_{\infty}=\sup_\Omega \left(|H_b(f)|_{\mathbb{H}^{m}}^2-4 V\right)$.

 \end{theorem} 
In the particular case where $(M,\theta)$ is  the Heisenberg group ${\mathbb{H}^{n}}$ endowed with the standard CR structure, this theorem provides an alternative  way to derive Corollary \ref{cor2}\\

The following observation will be crucial for the proof of Theorem \ref{them2}.
\begin{proposition}\label{pro1}
Let $(M,\theta)$ be a strictly pseudoconvex CR manifold and let
\begin{eqnarray*}
  f&:&(M,\theta) \longrightarrow \mathbb{H}^m\simeq \mathbb{C}^m\times \mathbb{R} \\
  &&x\quad \longrightarrow\quad f(x)=(F_1(x),...,F_m(x),\alpha(x))
\end{eqnarray*}
be a $C^2$ map such that $df(H(M))\subset H(\mathbb{H}^m)$. Then
$$H_b(f)=\sum_{j=1}^m(\Delta_b \varphi_j X_j+\Delta_b \psi_j Y_j)$$
where $\varphi_j(x)=\mbox{Re} F_j(x)$ and $\psi_j(x)=\mbox{Im} F_j(x)$.\\
In particular, $H_b (f)$ is a horizontal vector field and
\begin{equation*}
    |H_b (f)|^2_{\mathbb{H}^m}=4\sum_{j=1}^m [(\Delta_b \varphi_j)^2+(\Delta_b \psi_j)^2].
\end{equation*}
\end{proposition}
\begin{proof} One has, for any vector $W\in TM$,
\begin{equation*}
    df(W)=\sum_{j=1}^m \left(d\varphi_j(W)X_j+d\psi_j(W)Y_j\right)+\theta(df(W))T.
\end{equation*}
For $W\in H(M)$,  $df(W)\in H(\mathbb{H}^m)$ and, then,
\begin{equation}\label{equa10'}
    df(W)=\sum_{j=1}^m\left(d\varphi_j(W)X_j+d\psi_j(W)Y_j\right).
\end{equation}
Let $\{e_i\}$ be a local orthonormal frame of $H(M)$, then
\begin{equation*}
    \beta_f(e_i,e_i)=\nabla^f_{e_i}df(e_i)-df(\nabla_{e_i}e_i).
\end{equation*}
Since $e_i$ and $\nabla_{e_i}e_i$ are horizontal and that $df(H(M))\subset H(\mathbb{H}^m),$ we have
\begin{equation*}
    \beta_f(e_i,e_i)=\sum_{j=1}^m \nabla_{e_i}^f(d\varphi_j(e_i)X_j+d\psi_j(e_i)Y_j)-\sum_{j=1}^m[ d\varphi_j(\nabla_{e_i}e_i)X_j+d\psi_j(\nabla_{e_i}e_i)Y_j]
\end{equation*}
with
$$\nabla_{e_i}^f(d\varphi_j(e_i)X_j)=e_i\cdot d\varphi_j(e_i)X_j+d\varphi_j(e_i)\nabla^{\mathbb{H}^m}_{df(e_i)}X_j$$
and
$$\nabla_{e_i}^f(d\psi_j(e_i)Y_j)=e_i\cdot d\psi_j(e_i)Y_j+d\psi_j(e_i)\nabla^{\mathbb{H}^m}_{df(e_i)}Y_j.$$
Therefore, 
\begin{eqnarray}\label{equa10}
  \nonumber\beta_f(e_i,e_i) &=& \sum_{j=1}^m\left[e_i\cdot  d\varphi_j(e_i)-d\varphi_j(\nabla_{e_i}e_i)\right]X_j+\sum_{j=1}^m\left[e_i\cdot d\psi_j(e_i)-d\psi_j(\nabla_{e_i}e_i)\right]Y_j \\
   &+&\sum_{j=1}^m \left[d\varphi_j(e_i)\nabla^{\mathbb{H}^m}_{df(e_i)}X_j +d\psi_j(e_i)\nabla^{\mathbb{H}^m}_{df(e_i)}Y_j\right].
\end{eqnarray}
Recall that the Levi-Civita connection of $\mathbb{H}^m$ is such that
$$\nabla^{\mathbb{H}^m}_{X_k}X_j=\nabla^{\mathbb{H}^m}_{Y_k}Y_j=\nabla^{\mathbb{H}^m}_T T=0,$$
$$\nabla^{\mathbb{H}^m}_{X_k}Y_j=-2\delta_{kj}T,\;\;\quad\nabla^{\mathbb{H}^m}_{X_k} T=2Y_k,\;\;\quad \nabla^{\mathbb{H}^m}_{Y_k}T=-2X_k,$$
$$\nabla^{\mathbb{H}^m}_{Y_k}X_j=2\delta_{kj}T,\;\;\quad\nabla^{\mathbb{H}^m}_{T} X_k=2Y_k,\;\;\quad \nabla^{\mathbb{H}^m}_{T}Y_k=-2X_k.$$
Thus,
\begin{eqnarray*}
  \nabla_{df(e_i)}^{\mathbb{H}^m}X_j &=& \sum_k(d\varphi_k(e_i)\nabla_{X_k}X_j+d\psi_k(e_i)\nabla_{Y_k}X_j) \\
   &=& d\psi_j(e_i)\nabla_{Y_j}X_j=2d\psi_j(e_i)T.
\end{eqnarray*}
and
\begin{equation*}
    \nabla_{df(e_i)}^{\mathbb{H}^m}Y_j=-2d\varphi_j(e_i)T.
\end{equation*}
Replacing into \eqref{equa10} and summing up with respect to $i$, we get
\begin{eqnarray*}
  H_b(f) &=& \sum_{i=1}^{2n} \sum_{j=1}^{m} \left([e_i\cdot d\varphi_j(e_i)-d\varphi_j(\nabla_{e_i}e_i)]X_j + [e_i \cdot d\psi_j(e_i)-d\psi_j(\nabla_{e_i}e_i)]Y_j\right) \\
   &=&  \sum_{j=1}^{m}\left(\Delta_b \varphi_j X_j+\Delta_b \psi_j Y_j\right).
\end{eqnarray*}
\end{proof}

\begin{proof}[Proof of Theorem \ref{them2}]
As in the proof of Theorem \ref{them1}, we will use the components of the map $f$ as multiplication operators. Let us write $f(x)=(F_1(x),...,F_m(x),\alpha(x))\in \mathbb{C}^m\times \mathbb{R}$ and
$F_j(x)=\varphi_j(x)+i\psi_j(x)$. The main difference with respect to the Euclidean case is that here, only the   $\mathbb{C}^m$ components of $f$ come in. 
 All along this proof we will use the fact that, $\forall\,W\in H_x(M)$, the vector  $df(W)$ is horizontal and (see \eqref{equa10'})
\begin{equation}\label{11}
    |df(W)|^2_{\mathbb{H}^m}=4\sum_{j=1}^{m}\left(|d\varphi_j(W)|^2+|d\psi_j(W)|^2\right).
\end{equation}
Repeating the same calculations as in the proof of the Theorem \ref{them1}, we get
\begin{eqnarray}
    \nonumber \sum_{j=1}^{m}\langle  [-\Delta_b+V,\varphi_j] u_i,\varphi_j  u_i\rangle_{L^2}&+&\langle  [-\Delta_b+V,\psi_j] u_i,\psi_j  u_i\rangle_{L^2}\\
  \nonumber   &=& \sum_{j=1}^{m}\int_{\Omega}\big\{|\nabla^H  \varphi_j|_{G_{\theta}}^2+|\nabla^H  \psi_j|_{G_{\theta}}^2\big\}u_i^2.
\end{eqnarray}
 Let  $\{e_i\}$ be a $G_\theta$-orthonormal basis of $H_x(M)$, then
\begin{eqnarray}\label{equa12}
 \nonumber \sum_{j=1}^{m}|\nabla^H \varphi_j|^2_{G_{\theta}}+|\nabla^H \psi_j|^2_{G_{\theta}} &=& \sum_{j=1}^{m}\sum_{i=1}^{2n} \langle \nabla^H \nonumber\varphi_j,e_i\rangle^2_{G_{\theta}}+\langle \nabla^H \psi_j,e_i\rangle^2_{G_{\theta}}\\
   \nonumber&=&  \sum_{i=1}^{2n} \sum_{j=1}^{m}\langle \nabla \varphi_j,e_i\rangle^2_{G_{\theta}}+\langle \nabla \psi_j,e_i\rangle^2_{G_{\theta}}\\
   \nonumber&=&  \sum_{i=1}^{2n} \sum_{j=1}^{2m} (d\varphi_j(e_i)^2+d\psi_j(e_i)^2)\\
  \nonumber &=& \frac 14 \sum_{i=1}^{2n}| df(e_i)|^2_{\mathbb{H}^m}=\frac n2.
\end{eqnarray}
Thus, 
\begin{equation}\label{equa12'}
\sum_{j=1}^{m}\langle  [-\Delta_b+V,\varphi_j] u_i,\varphi_j  u_i\rangle_{L^2}+\langle  [-\Delta_b+V,\psi_j] u_i,\psi_j  u_i\rangle_{L^2}=\frac n2.
\end{equation}
On the other hand,
\begin{eqnarray*}
  \|[-\Delta_b+V,\varphi_j]u_i\|^2_{L^2} &=&  \int_{\Omega}\left((\Delta_b \varphi_j)u_i+2\langle \nabla^H  \varphi_j,\nabla^H  u_i\rangle_{G_{\theta}}\right)^2 \\
   &=&  \int_{\Omega}(\Delta_b \varphi_j)^2 u_i^2+4\int_{\Omega}\langle \nabla^H  \varphi_j,\nabla ^H  u_i\rangle^2_{G_{\theta}}\\
   &+& 2 \int_{\Omega}(\Delta_b \varphi_j)\langle \nabla^H  \varphi_j,\nabla^H  u_i^2\rangle_{G_{\theta}}.
\end{eqnarray*}
We have a similar formula for $\|[-\Delta_b+V,\psi_j]u_i\|^2_{L^2}$. Since $\nabla^H u_i\in H(M)$, one has
\begin{eqnarray*}
   \sum_{j=1}^{m}\langle \nabla^H \varphi_j,\nabla ^H  u_i\rangle^2_{G_{\theta}}&+&\langle \nabla^H \psi_j,\nabla ^H  u_i\rangle^2_{G_{\theta}}\\ &=& \sum_{j=1}^{m}\{d\varphi_j(\nabla^H u_i)^2 +d\psi_j(\nabla^H u_i)^2\}\\
   &=& \frac 14|df(\nabla^H u_i)^2|_{\mathbb{H}^m} 
   =\frac 14 |\nabla^H u_i|^2_{G_{\theta}}.
\end{eqnarray*}
 Therefore,
\begin{eqnarray*}
\sum_{j=1}^{m} \int_{\Omega} \left(\langle \nabla^H \varphi_j,\nabla ^H  u_i\rangle^2_{G_{\theta}}+\langle \nabla^H \psi_j,\nabla ^H  u_i\rangle^2_{G_{\theta}}\right)&=&\frac 14 \int_{\Omega} |\nabla^H u_i|^2_{G_{\theta}}\\
&=& \frac14\left( \lambda_i-\int_{\Omega}Vu_i^2\right).
\end{eqnarray*}
For the  two remaining terms, we have thanks to Proposition \ref{pro1} and the identity \eqref{equa9'},
\begin{equation*}
    \sum_{j=1}^{m}\int_{\Omega} \left((\Delta_b \varphi_j)^2 +(\Delta_b \psi_j)^2 \right)u_i^2=\frac 14\int_{\Omega} |H_b(f)|_{\mathbb{H}^m}^2 u_i^2
\end{equation*}
and
\begin{eqnarray*}
 \sum_{j=1}^{m}\int_{\Omega} & &\left(\Delta_b \varphi_j \langle \nabla^H \varphi_j,\nabla^Hu_i^2\rangle_{G_{\theta}}+\Delta_b \psi_j\langle \nabla^H \psi_j,\nabla^Hu_i^2\rangle_{G_{\theta}}\right)\\
 &=&\frac 14 \int_{\Omega} \langle H_b(f),\sum_{j=1}^{m} d\varphi_j(\nabla^H u_i^2)X_j+\sum_{j=1}^{m} d\psi_j(\nabla^H u_i^2)Y_j\rangle_{\mathbb{H}^m}\\
   &=& \frac 14\int_{\Omega} \langle H_b(f),df(\nabla^H u_i^2)\rangle_{\mathbb{H}^m}=0,
\end{eqnarray*}
where the last equality follows from the fact that $H_b(f)$ is orthogonal to $df(H(M))$ (Lemma \ref{lem1}).
Finally, 
\begin{equation}\label{equa12''}
\|[-\Delta_b+V,\varphi_j]u_i\|^2_{L^2}+ \|[-\Delta_b+V,\psi_j]u_i\|^2_{L^2}=\lambda_i+\frac 14\int_{\Omega} \left(|H_b(f)|_{\mathbb{H}^m}^2 -V\right)u_i^2.
\end{equation}
Applying  Lemma \ref{lma2}  with $A=-\Delta_b+V$ and $B=\varphi_j$ then $B=\psi_j$, summing up with respect to $j$ and  using \eqref{equa12'} and \eqref{equa12''}, we obtain the inequality \eqref{eqheis1}.

As in the proof of Theorem  \ref{them1}, we derive the inequalities   \eqref{eqheis2} and  \eqref{eqheis3} from \eqref{eqheis1} with $p=2$. 
\end{proof}
\section{Reilly type inequalities for CR manifolds mapped into the Euclidean space or the Heisenberg group }

Let $(M,\theta)$ be a compact strictly pseudo-convex CR manifold. If  $f:(M,{\theta})\longrightarrow \mathbb{R}^m$ is a semi-isometric $C^2$ map, then Theorem \ref{them1} (i.e. inequality \eqref{eqeuc1} with $k=1$ and $p=1$) gives,
$$\lambda_2(-\Delta_b+V)\le (1+\frac 2n) \lambda_1(-\Delta_b+V)+\frac1{2n}\int_M\left(|H_b(f)|_{\mathbb{R}^{m}}^2-4V\right)u_1^2.$$

When $M$ is a compact manifold without boundary and $V=0$, one has $\lambda_1(-\Delta_b)=0$ and $u_1^2=\frac 1 {V(M,\theta)}$. Therefore,  the following Reilly type result holds (see\cite{SoufiIlias} for details about Reilly inequalities)
\begin{equation*}
    \lambda_2(-\Delta_b)\leq \frac{1}{2nV(M,\theta)}\int_M |H_b(f)|_{\R^m}^2.
  \end{equation*}
  
 This result can be obtained in an independent and simpler way, in the spirit of Reilly's proof, under  weaker assumptions on $f$.   Moreover, the equality case can be characterized. Indeed, we first have the following
\begin{theorem}\label{them3}
 Let $(M,\theta)$ be a compact strictly pseudoconvex CR manifold of dimension $2n+1$ without boundary. For every $C^2$  map $f:(M,\theta)\longrightarrow \mathbb{R}^m$  one has  
  \begin{equation}\label{reilly}
    \lambda_2(-\Delta_b)E_b(f) \leq \frac 12\int_M |H_b(f)|_{\R^m}^2
  \end{equation}
  where the equality holds if and only if the Euclidean components $f_1, \dots, f_m$ of $f$ satisfy $-\Delta_b f_\alpha=\lambda_2(-\Delta_b)\left(f_\alpha  - \fint f_\alpha\right)$ for every $\alpha\le m$.  
  
\end{theorem}
\begin{proof} 
  Replacing if necessary $f_\alpha$ by $f_\alpha  - \fint f_\alpha$ we can assume without loss of generality  that the Euclidean components $f_1, \dots, f_m$ of $f$ satisfy $\int_M f_\alpha \vartheta_{\theta}=0 $  so that, we have
       \begin{equation}\label{equa13}
       \lambda_2(-\Delta_b) \int_M f_\alpha^2\leq \int_M |\nabla^H f_\alpha|_{G_\theta}^2.
       \end{equation}
Summing up with respect to $\alpha$, we get
    \begin{equation*}
        \lambda_2(-\Delta_b)\int_M |f|_{\R^m}^2\leq\int_M \sum_{\alpha=1}^m |\nabla^H f_\alpha|_{G_\theta}^2.
    \end{equation*}
    Denoting by $\{\epsilon_\alpha\}$ the standard  basis of $\mathbb{R}^m$ and by  $\{X_i\}$ a local orthonormal frame of $H(M)$,   we observe that
    \begin{eqnarray*}
    2 e_b(f)&=& \sum_{i=1}^{2n}|df(X_i)|_{\R^m}^2=\sum_{i=1}^{2n} \sum_{\alpha=1}^m\langle df(X_i),\epsilon_\alpha\rangle_{\R^m}^2\\
    &=&\sum_{\alpha=1}^m\sum_{i=1}^{2n}|df_\alpha(X_i)|_{\R^m}^2=\sum_{\alpha=1}^m |\nabla^H f_\alpha|_{G_\theta}^2.
     \end{eqnarray*}
    Therefore,
    \begin{equation}\label{equa13'}
        \lambda_2(-\Delta_b)\int_M |f|_{\R^m}^2\leq\int_M \sum_{\alpha=1}^m |\nabla^H f_\alpha|_{G_\theta}^2= 2E_b(f).
    \end{equation} 
 On the other hand,  we have
   \begin{eqnarray*}
    4E_b(f)^2 &=& \left(\sum_{\alpha=1}^m\int_M|\nabla^H f_\alpha|_{G_\theta}^2\right)^2 
      = \left(\sum_{\alpha=1}^m \int_M f_\alpha \Delta_b f_\alpha\right)^2\\
      &=& \left(\int_M\langle f(x),\sum_\alpha^m (\Delta_b f_\alpha)\epsilon_\alpha\rangle_{\R^m}\right)^2 \\
      &=& \left(\int_M\langle f(x),H_b(f)\rangle_{\R^m}\right)^2
      \leq \int_M |f|_{\R^m}^2\int_M |H_b(f)|_{\R^m}^2.
   \end{eqnarray*}
   Combining with \eqref{equa13'}, we get
   \begin{equation*}
   4E_b(f)^2\leq \frac{ 2E_b(f)}{\lambda_2(-\Delta_b)}\int_M |H_b(f)|_{\R^m}^2
   \end{equation*}
   which gives the desired inequality.
   
   Now, if we have, for every $\alpha\le m$,  $-\Delta_b f_\alpha=\lambda_2(-\Delta_b)f_\alpha$, then $H_b(f)= (\Delta_b f_1,\dots,\Delta_bf_m)=-\lambda_2(-\Delta_b) f	$ and $\int_M |H_b(f)|_{\R^m}^2 =\lambda_2(-\Delta_b)^2\int_M |f|_{\R^m}^2 $. On the other hand,  $E_b(f)= \int_M \sum_{\alpha=1}^m |\nabla^H f_\alpha|_{G_\theta}^2=\lambda_2(-\Delta_b)\int_M |f|_{\R^m}^2 $ which implies that the equality holds in \eqref{reilly}. Reciprocally, if  the equality holds in \eqref{reilly} for a nonconstant map $f$, then it also holds in \eqref{equa13} for each $\alpha$. Thus, the functions $f_1,\dots, f_m$ belong to the $\lambda_2(-\Delta_b)$-eigenspace of $-\Delta_b$. 
 \end{proof}

If a map $f:(M,\theta)\longrightarrow \mathbb{R}^m$ preserves the metric with respect to horizontal directions (i.e., $|df(X)|_{\R^m}= |X|_{G_\theta}$ for any $X\in H(M)$), then  its energy density $e_b(f)$ is constant equal to $n$ and  
$$E_b(f)=nV(M,\theta).$$
Inequality \eqref{reilly} becomes in this case
\begin{equation}\label{reilly iso}
    \lambda_2(-\Delta_b)\leq \frac{1}{2nV(M,\theta)}\int_M |H_b(f)|_{\R^m}^2.
  \end{equation}
  The characterization of the equality case is the last inequality requires the following Takahashi's type result.
\begin{lemma}\label{takahashi} 
 Let $(M,\theta)$ be a strictly pseudoconvex CR manifold of dimension $2n+1$  and let $f:(M,\theta)\longrightarrow \mathbb{R}^m$  be  $C^2$ map.
\begin{itemize}
 \item[i)] Assume that $f(M)$ is contained in a sphere $\mathbb{S}^{m-1}(r)$ of radius $r$ centered at the origin. Then $f$ is pseudo-harmonic from $(M,\theta)$ to $S^{m-1}(r)$ if and only if its Euclidean components $f_1, \dots, f_m$ satisfy, $\forall \alpha\le m$,
 $$-\Delta_b f_\alpha=\mu f_\alpha $$
 with $\mu = \frac {2}{r^2}e_b(f) \in C^\infty(M)$.  
\item[ii)]  Assume that $f$ is semi-isometric. If the Euclidean components $f_1, \dots, f_m$ of $f$ satisfy, $\forall \alpha\le m$, $-\Delta_b f_\alpha=\lambda f_\alpha $, for some  $\lambda\in\R$, then $f(M)$ is contained in the sphere $\mathbb{S}^{m-1}(r)$ of radius $r=\sqrt{\frac {2n}{\lambda} }$ and $f$ is a pseudo-harmonic map from $(M,\theta)$ to $S^{m-1}(r)$. 
Conversely, if $f(M)$ is contained in a sphere $\mathbb{S}^{m-1}(r)$ and if $f$ is a pseudo-harmonic map from $(M,\theta)$ to  $S^{m-1}(r)$, then, $\forall \alpha\le m$, $-\Delta_b f_\alpha=\frac{2n}{r^2} f_\alpha $.
\end{itemize}
 \end{lemma}
 This lemma is to be compared with Example 5.3 of \cite{BarlettaDragomirUrakawa} in which a sign mistake in Greenleaf's formula led to an incorrect characterization of pseudo-harmonic maps into spheres. 
\begin{proof}[Proof of Lemma \ref{takahashi}]
i) For convenience, let us write $f=j\circ \bar f$ where $j:\mathbb{S}^{m-1}(r)\to \R^m$ is the standard embedding and $\bar f:M\to \mathbb{S}^{m-1}(r)$ is defined by $\bar f(x)=f(x)$. It is straightforward to observe that, $\forall X$, $Y\in H(M)$,
$$\beta_f (X,Y) =B_j(d\bar f(X),d\bar f(Y)) +dj(\beta_{\bar f} (X,Y))$$
where $B_j (W,W)=-\frac 1{r^2}|W|_{\R^m}^2\vec{x}$ is the second fundamental form of the sphere $\mathbb{S}^{m-1}(r)$. Taking the trace, we obtain
$$H_b(f)=-\frac{2e_b(\bar f)}{r^2} \bar f +dj(H_b(\bar f))=-\frac{2e_b(f)}{r^2} f +dj(H_b(\bar f)).$$
Hence, if $f$ is pseudo-harmonic from $(M,\theta)$ to $S^{m-1}(r)$, then $H_b(\bar f)=0$ and, consequently, $H_b(f)=-\frac{2e_b(f)}{r^2} f   $ with $H_b(f)=(\Delta_b f_1, \dots, \Delta_b f_m)$ (see \eqref{equa1}). Thus,  $\forall \alpha\le m$, $-\Delta_b f_\alpha=\frac {2}{r^2}e_b(f) f_\alpha$.
  \\
 
Reciprocally, if there exists a function $\mu  \in C^\infty(M)$ such that  $-\Delta_b f_\alpha=\mu f_\alpha $ for every $\alpha\le m$, then 
$$0=\Delta_b \left(\sum_{\alpha=1}^m f_\alpha^2 \right)=-2\mu \sum_{\alpha=1}^m f_\alpha^2+2\sum_{\alpha=1}^m |\nabla^H f_\alpha|_{G_\theta}^2=-2{\mu}{r^2} + 4e_b(f).$$
Hence, $\mu = \frac{2e_b(f)}{r^2}$, $H_b(f)=-\frac{2e_b(f)}{r^2} f $ and, then, $H_b(\bar f)=0$, which means that $f$ is pseudo-harmonic from $(M,\theta)$ to $S^{m-1}(r)$.

\medskip

\noindent ii) From the assumptions, one has $H_b(f)=-\lambda f $ (see \eqref{equa1}). Since $f$ is semi-isometric, we know that $H_b(f)$ is orthogonal to $df(H(M))$ (Lemma \ref{lem1}). Therefore, $\forall x\in M$ and $\forall X\in H_x(M)$, one has
$\left\langle f(x), df_x(X) \right\rangle_{\R^m} =0$
which implies that the function $x\mapsto |f(x)|_{\R^m}^2$ has zero derivative with respect to all horizontal directions. Since the distribution  $H(M)$ is not integrable, this implies that $|f(x)|_{\R^m}^2$ is constant on $M$, that is $f(M)$ is contained in a sphere $\mathbb{S}^{m-1}(r)$ of radius $r$ centered at the origin. The pseudo-harmonicity of $f$ from $M$ into $\mathbb{S}^{m-1}(r)$ then follows from (i). Moreover, one necessarily has $\lambda= \frac{2e_b(f)}{r^2}$ with $e_b(f)=n$ since $f$ is semi-isometric. Thus, the radius of the sphere is such that $r^2=\frac {2n}{\lambda}$.
 
\end{proof}

Theorem \ref{them3} and Lemma \ref{takahashi}  lead to the following 
\begin{corollary}\label{them3'}
 Let $(M,\theta)$ be a compact strictly pseudoconvex CR manifold of dimension $2n+1$ without boundary and let $f:(M,\theta)\longrightarrow \mathbb{R}^m$  be  $C^2$ semi-isometric map.
  Then
  \begin{equation}\label{reilly iso}
    \lambda_2(-\Delta_b)\leq \frac{1}{2nV(M,\theta)}\int_M |H_b(f)|_{\R^m}^2.
  \end{equation}
Moreover, the equality holds in this inequality if and only if $f(M)$ is contained in a sphere $\mathbb{S}^{m-1}(r)$ of radius $r=\sqrt{\frac {2n}{\lambda_2(-\Delta_b)} }$ and $f$ is a pseudo-harmonic map from $(M,\theta)$ to the sphere $S^{m-1}(r)$. 
\end{corollary}

Similarly, for CR manifolds mapped into the Heisenberg group, one has the following

 \begin{theorem}\label{them4'}
 Let $(M,\theta)$ be a compact strictly pseudoconvex CR manifold of dimension $2n+1$ without boundary.\\ 
 i) Let $f:M\longrightarrow \mathbb{H}^{m}=\R^{2m}\times \R$  be any  $C^2$ map satisfying  $df (H(M))\subseteq H(\mathbb{H}^{m})$. Then   
    $$\lambda_2(-\Delta_b) E_b(f)\leq \frac{1}{2}\int_M |H_b(f)|_{\mathbb{H}^m}^2$$
    where the equality holds if and only if the first $2m$ components $f_1, \dots, f_{2m}$ of $f$ satisfy $-\Delta_b f_\alpha=\lambda_2(-\Delta_b)\left(f_\alpha  - \fint f_\alpha\right)$ for every $\alpha\le 2m$.\\  
 ii) Let $f:M\longrightarrow \mathbb{H}^{m}$  be any  $C^2$ semi-isometric map satisfying  $df (H(M))\subseteq H(\mathbb{H}^{m})$. Then  
    $$\lambda_2(-\Delta_b)\leq \frac{1}{2nV(M,\theta)}\int_M |H_b(f)|_{\mathbb{H}^m}^2.$$
 Moreover, the equality holds in this last inequality if and only if $f(M)$ is contained in the product $\mathbb{S}^{2m-1}(r)\times \R\subset \mathbb{H}^{m}$ with $r=\sqrt{\frac {2n}{\lambda_2(-\Delta_b)} }$, and $\pi\circ f$ is a pseudo-harmonic map from $(M,\theta)$ to the sphere $S^{2m-1}(r)$, where $\pi: \mathbb{H}^{m}\to\R^{2m}$ is the standard projection. 
 
\end{theorem}

\begin{proof}
i) Let $f:M\longrightarrow \mathbb{H}^{m}=\R^{2m}\times \R$  be a  $C^2$ map satisfying  $df (H(M))\subseteq H(\mathbb{H}^{m})$ and set $\tilde f:= \pi\circ f:M\longrightarrow \R^{2m}$ where $\pi: \mathbb{H}^{m}\to\R^{2m}$ is the standard projection. One has, for every pair $(X,Y)$ of horizontal vectors,
 $$\beta_{\tilde f}(X,Y)= \beta_\pi(df(X), df(Y))+d\pi (\beta_f(X,Y)).$$
Since for any $X\in H(\mathbb{H}^{m})$, $|d\pi(X)|_{\mathbb{R}^{2m}}^2=\frac 14|X|_{\mathbb{H}^{m}}^2$ (see \eqref{equa9'}) and $d\pi(T)=0$,  one can easily check that  $\beta_\pi\equiv0$ (Corollary \ref{cor0}) and, then, $\beta_{\tilde f}(X,Y)=d\pi (\beta_f(X,Y))$. Thus, $H_b(\tilde f)=d\pi(H_b(f))$ and, since $H_b(f)$ is horizontal (Proposition \ref{pro1}), $ |H_b(\tilde f)|_{\mathbb{R}^{2m}}^2=\frac 14|H_b(f)|_{\mathbb{H}^{m}}^2$. On the other hand, it is clear that  $e_b(\tilde f)=\frac 14 e_b(f)$ and, then, $E_b(\tilde f)=\frac 14 E_b(f)$. Therefore, it suffices to apply Theorem \ref{reilly} to complete the proof of the first part of the theorem. 

\medskip

\noindent ii) Assume now that the map $f$ is semi-isometric. Using the assumption that $f$ preserves horizontality, i.e., $df (H(M))\subseteq H(\mathbb{H}^{m})$, one checks that  the map $ 2 \pi\circ f$  is also semi-isometric. Applying Corollary \ref{them3'} to the latter we easily deduce what is stated in  part (ii) of the theorem.

\end{proof}

\section{Eigenvalues of the Horizontal Laplacian on a Carnot group}
A Carnot group of step $r$ is a connected, simply connected, nilpotent Lie group $G$ whose Lie algebra $\mathfrak{g}$ admits a stratification $$\mathfrak{g}=V_1\oplus...\oplus V_r$$ so that $[V_1,V_j]=V_{j+1},\;j=1,...,r-1$ and $[V_i,V_j]\subset V_{i+j},\;j=1,...,r$, with $V_k=\{0\}$ for $k>r$. We also assume  that $\mathfrak{g}$ carries   a scalar product $\left\langle ,\right\rangle_\mathfrak{g}$ for which the subspaces $V_j$ are mutually orthogonal. 
The layer $V_1$ generates the whole $\mathfrak{g}$ and induces a sub-bundle $HG$ of $TG$ of rank $d_1=\dim V_1$ that we call the horizontal bundle of the Carnot group. The Heisenberg group $\mathbb{H}^{d}$ is  the simplest example of a Carnot group of step 2. 

\medskip

For each $i\le r$, let $\{e^i_1,\cdots, e^i_{d_i}\}$ be an orthonormal basis of $V_i$ and denote by $\{X^i_1,\cdots, X^i_{d_i}\}$ the system of left invariant vector fields that coincides with $\{e^i_1,\cdots, e^i_{d_i}\}$ at the identity element of $G$. We consider the Riemannian metric $g_G$ on $G$ with  respect to
which the family $\{X^1_1,\cdots, X^1_{d_1}, \cdots, X^r_1,\cdots, X^r_{d_r}\}$
constitute an orthonormal frame for $TG$. The   corresponding Levi-Civita connection  $\nabla$ induces a connection  $\nabla^H$ on $HG$ that we call ``horizontal connection" : If X  and $Y$ are smooth sections of $HG$, then $\nabla^H_XY=\pi_H \nabla_XY$, where $\pi_H:TG\to HG$ is the orthogonal projection.  The horizontal Laplacian $\Delta_H $ is then defined for every $C^2$ function on $G$ by
$$\Delta_H u:=\mbox{trace}_H\nabla^Hdu=\sum_{i\le d_1}X^1_i\cdot \left(X^1_i\cdot u\right),$$
where the last equality follows from the fact that $\nabla^H_{X^i_1} X^j_1=0$ for any $i,j=1\dots d_1$.
The operator  $\Delta_H $ is a hypoelliptic operator of Hörmander type.

\begin{theorem}
Let $G$ be a Carnot group and let  $\Omega$ be a bounded domain in $G$. Let $V$ be a function on $\Omega$ so that the  operator $-\Delta_H +V$,  with Dirichlet boundary conditions if  $\Omega \neq G$, admits a purely discrete spectrum $\{\lambda_j\}_{j\geq1}$ which is bounded from below. Then, for every $k\geq 1$ and $p\in \R,$
  \begin{equation*}
    \sum_{i=1}^k\big(\lambda_{k+1}-\lambda_i\big)^p \leq \frac {\max \{4,2p\}}d \sum_{i=1}^k \big(\lambda_{k+1}-\lambda_i\big)^{p-1}\big(\lambda_i -T_i\big),
  \end{equation*}
 where $d$ is the rank of the horizontal distribution $HG$,  $T_i=\int_\Omega V u_i^2 v_G$ and $v_G$ is the Riemannian volume element associated with $g_G$. Moreover, if $V$ is bounded below on $\Omega$, then for every  $k\geq 1$, 
$$\lambda_{k+1}\leq \left(1+\frac{4}{d}\right)\frac{1}{k}\sum_{i=1}^k \lambda_i-\frac 4d \inf_{\Omega}V$$
and
 $$\lambda_{k+1}\leq \left(1+\frac{4}{d}\right)k^{\frac{2}{d}}\lambda_1- C(d,k)\inf_\Omega V$$
 with $C(d,k)= (1+\frac{4}{d})k^{\frac{2}{d}}-1.$
\end{theorem}
\begin{proof}
Let $\{e_1,\dots,e_d\}$ be an orthonormal basis of the subspace $V_1$ and denote by $\{X_1,\cdots, X_{d}\}$ the system of left invariant vector fields that coincides with $\{e_1,\dots,e_d\}$ at the identity element of $G$.  Since the group $G$ is nilpotent, the exponential map $\exp :\mathfrak{g}\longrightarrow G$ is a global diffeomorphism. We can define, for each $i\le d$, a smooth map $x_i:G\to\R$ by
$$x_i(g):=\left\langle\exp^{-1}(g),e_i\right\rangle_\mathfrak{g}.$$
These functions satisfy (see \cite[Proposition 5.7]{danielli-garofalo-nhieu}), $\forall i,j=1,...,m$,
 $$X_j\cdot x_i=\delta_{ij} \; \; \mbox{and} \; \; \Delta_H x_i=0.$$

Again, we apply Lemma \ref{lma2} with $A=-\Delta_H+V$ and $B=x_\alpha$, $1\leq\alpha\leq m$. We need to deal with the calculation of $\left\langle [-\Delta_H+V,x_\alpha]u_i,x_\alpha u_i\right\rangle_{L^2}$ and $\left\|[-\Delta_H+V,x_\alpha]u_i\right\|^2_{L^2}$, where $\{u_i\}_{i\ge1}$ a complete orthonormal family of eigenfunctions  with $(-\Delta_b +V) u_i=\lambda_i u_i$. We have after a straightforward calculation :
$$ [-\Delta_H+V,x_\alpha]u_i=-2X_\alpha\cdot u_i.$$
Integrating by parts we get
$$\int_\Omega \left(X_\alpha\cdot u_i\right)x_\alpha u_i=\frac 12\int_\Omega \left(X_\alpha\cdot u^2_i\right)x_\alpha =-\frac 12\int_\Omega u_i^2 \left(X_\alpha\cdot x_\alpha\right)=-\frac 12\int_\Omega u_i^2 =-\frac 12.$$
Thus, 
\begin{eqnarray*}
      \sum_{\alpha=1}^d \left\langle [-\Delta_H+V,x_\alpha]u_i,x_\alpha u_i\right\rangle_{L^2}=-2 \sum_{\alpha=1}^d \int_\Omega \left(X_\alpha\cdot u_i\right)x_\alpha u_i=d .
   \end{eqnarray*}  
On the other hand, we have 
$$\sum_{\alpha=1}^d\left\|[-\Delta_H+V,x_\alpha]u_i\right\|^2_{L^2}=4\sum_{\alpha=1}^d\int_\Omega \left|X_\alpha \cdot u_i\right|^2=4\left(\lambda_i-T_i\right) $$
Putting these identities in Lemma $3.1$, we obtain the first inequality of the theorem.\\

The rest of the proof is identical to that of Theorem \ref{them1}.

\end{proof}

\bibliographystyle{plain}
\bibliography{biblio}

\end{document}